\documentclass[11pt,a4paper]{amsart}

\usepackage{mathtools}
\usepackage{amsthm}
\usepackage{amssymb}
\usepackage[
colorlinks,
citecolor=cyan,
]{hyperref}

\newcommand*\fullref[3][\relax]{%
  \ifdefined\hyperref%
    {\hyperref[#3]{#2\penalty 200\ \ref*{#3}#1}}%
  \else%
    {#2\penalty 200\ \relax\ref{#3}#1}%
  \fi%
}

\usepackage[dvipsnames,table]{xcolor}

\usepackage{ytableau}
\ytableausetup{smalltableaux}

\usepackage{tikz}
\usetikzlibrary{calc}

\usepackage{etoolbox}

\usetikzlibrary{arrows.meta}

\tikzset{
  normalarrow/.style={line width=.6pt},
}

\tikzset{
  normalarrowlabel/.style={
    auto,
    inner sep=.5mm,
    outer sep=0mm,
    font=\footnotesize,
  },
  tinyarrowlabel/.style={
    auto
    inner sep=.2mm,
    outer sep=0mm,
    font=\tiny,
  },
  arrowlabel/.style={
    normalarrowlabel
  },
  marrowlabel/.style={
    normalarrowlabel,
  },
  crystaledges/.style={
    f1/.style={
      ->,
      normalarrow,
      labelled/.style={every to/.style={edge node={node[marrowlabel] {$1$}}}},
      coloured/.style={draw=red},
    },
    f2/.style={
      ->,
      normalarrow,
      labelled/.style={every to/.style={edge node={node[marrowlabel] {$2$}}}},
      coloured/.style={draw=blue},
    },
    f3/.style={
      ->,
      normalarrow,
      labelled/.style={every to/.style={edge node={node[marrowlabel] {$3$}}}},
      coloured/.style={draw=green!50!black},
    },
    f4/.style={
      ->,
      normalarrow,
      labelled/.style={every to/.style={edge node={node[marrowlabel] {$4$}}}},
      coloured/.style={draw=brown},
    },
    f5/.style={
      ->,
      normalarrow,
      labelled/.style={every to/.style={edge node={node[arrowlabel] {$5$}}}},
      coloured/.style={draw=cyan},
    },
    f6/.style={
      ->,
      normalarrow,
      labelled/.style={every to/.style={edge node={node[arrowlabel] {$6$}}}},
      coloured/.style={draw=violet},
    },
    f7/.style={
      ->,
      normalarrow,
      labelled/.style={every to/.style={edge node={node[arrowlabel] {$7$}}}},
      coloured/.style={draw=brown!50!black},
    },
    f8/.style={
      ->,
      normalarrow,
      labelled/.style={every to/.style={edge node={node[arrowlabel] {$7$}}}},
      coloured/.style={draw=magenta},
    },
    f9/.style={
      ->,
      normalarrow,
      labelled/.style={every to/.style={edge node={node[arrowlabel] {$7$}}}},
      coloured/.style={draw=brown!50!black},
    },
    f10/.style={
      ->,
      normalarrow,
      labelled/.style={every to/.style={edge node={node[arrowlabel] {$7$}}}},
      coloured/.style={draw=lime},
    },
    f11/.style={
      ->,
      normalarrow,
      labelled/.style={every to/.style={edge node={node[arrowlabel] {$7$}}}},
      coloured/.style={draw=olive},
    },
    f12/.style={
      ->,
      normalarrow,
      labelled/.style={every to/.style={edge node={node[arrowlabel] {$7$}}}},
      coloured/.style={draw=pink},
    },
    fi/.style={
      ->,
      normalarrow,
      labelled/.style={every to/.style={edge node={node[marrowlabel] {$i$}}}},
      coloured/.style={draw=gray},
    },
    fn2/.style={
      ->,
      normalarrow,
      labelled/.style={every to/.style={edge node={node[marrowlabel] {$n{-}2$}}}},
      coloured/.style={draw=brown},
    },
    fn1/.style={
      ->,
      normalarrow,
      labelled/.style={every to/.style={edge node={node[arrowlabel] {$n{-}1$}}}},
      coloured/.style={draw=cyan},
    },
    fn/.style={
      ->,
      normalarrow,
      labelled/.style={every to/.style={edge node={node[arrowlabel] {$n$}}}},
      coloured/.style={draw=violet},
    },
    df1/.style={
      f1,
      densely dotted,
    },
    df2/.style={
      f2,
      densely dotted,
    },
    df3/.style={
      f3,
      densely dotted,
    },
    dfi/.style={
      fi,
      densely dotted,
    },
    dfn2/.style={
      fn2,
      densely dotted,
    },
    dfn1/.style={
      fn1,
      densely dotted,
    },
    dfn/.style={
      fn,
      densely dotted,
    },
  },
  labelledcrystaledges/.style={
    crystaledges,
    f1/.append style={labelled},
    f2/.append style={labelled},
    f3/.append style={labelled},
    f4/.append style={labelled},
    f5/.append style={labelled},
    f6/.append style={labelled},
    f7/.append style={labelled},
    f8/.append style={labelled},
    f9/.append style={labelled},
    f10/.append style={labelled},
    f11/.append style={labelled},
    f12/.append style={labelled},
    fi/.append style={labelled},
    fn2/.append style={labelled},
    fn1/.append style={labelled},
    fn/.append style={labelled},
  },
  colouredcrystaledges/.style={
    crystaledges,
    f1/.append style={coloured},
    f2/.append style={coloured},
    f3/.append style={coloured},
    f4/.append style={coloured},
    f5/.append style={coloured},
    f6/.append style={coloured},
    f7/.append style={coloured},
    f8/.append style={coloured},
    f9/.append style={coloured},
    f10/.append style={coloured},
    f11/.append style={coloured},
    f12/.append style={coloured},
    fi/.append style={coloured},
    fn2/.append style={coloured},
    fn1/.append style={coloured},
    fn/.append style={coloured},
  },
  labelledcolouredcrystaledges/.style={
    crystaledges,
    f1/.append style={labelled,coloured},
    f2/.append style={labelled,coloured},
    f3/.append style={labelled,coloured},
    f4/.append style={labelled,coloured},
    f5/.append style={labelled,coloured},
    f6/.append style={labelled,coloured},
    f7/.append style={labelled,coloured},
    f8/.append style={labelled,coloured},
    f9/.append style={labelled,coloured},
    f10/.append style={labelled,coloured},
    f11/.append style={labelled,coloured},
    f12/.append style={labelled,coloured},
    fi/.append style={labelled,coloured},
    fn2/.append style={labelled,coloured},
    fn1/.append style={labelled,coloured},
    fn/.append style={labelled,coloured},
  },
  crystalvertex/.style={
    font=\small,
    inner sep=.5mm,
    outer sep=0mm,
  },
  smallcrystalvertex/.style={
    crystalvertex,
    font=\scriptsize,
  },
  bigcrystalvertex/.style={
    crystalvertex,
    inner sep=1mm,
    font=\normalsize,
  },
  crystal/.style={
    x=10mm,
    y=10mm,
    every node/.style={crystalvertex},
    labelledcrystaledges,
  },
  bigcrystal/.style={
    x=15mm,
    y=15mm,
    every node/.style={bigcrystalvertex},
    labelledcrystaledges,
  },
  smallcrystal/.style={
    x=7mm,
    y=7mm,
    every node/.style={smallcrystalvertex},
    colouredcrystaledges,
  },
}

\newcommand*\cedge[2][]{%
  \ifstrempty{#1}{
    \mathrel{%
    \begin{tikzpicture}[crystal,labelledcolouredcrystaledges,baseline=(start.base)]
      \node[inner sep=0mm] (start) at (0,0) {$\mathstrut$};
      \node[inner sep=0mm] (stop) at (1,0) {$\mathstrut$};
      \draw[f#2] (start) to (stop);
    \end{tikzpicture}%
    }%
  }{%
    \mathrel{%
    \begin{tikzpicture}[crystal,labelledcolouredcrystaledges,baseline=(start.base)]
      \node[inner sep=0mm] (start) at (0,0) {$\mathstrut$};
      \node[inner sep=0mm] (stop) at (1,0) {$\mathstrut$};
      \draw[f#2,every to/.style={edge node={node[marrowlabel] {$#1$}}}] (start) to (stop);
    \end{tikzpicture}%
    }%
  }%
}
\usetikzlibrary{matrix}

\tikzset{
  pretableaumatrix/.style={
    ampersand replacement=\&,
    matrix of math nodes,
    outer sep=1mm,
    inner sep=0mm,
    anchor=center,
    row sep={between borders,-\pgflinewidth},
    column sep={between borders,-\pgflinewidth},
    dottedentry/.style={densely dotted},
    dashedentry/.style={densely dashed},
    shadedentry/.style={fill=lightgray},
    darkshadedentry/.style={fill=gray},
    spaceentry/.style={draw=none,execute at begin node=\null},
    labelentry/.style={draw=none,execute at begin node=\null,font=\footnotesize},
  },
  pretableaunode/.style={
    font=\small,
    sharp corners,
    rectangle,
    anchor=base,
    text height=3.75mm,
    text depth=1.25mm,
    minimum height=5mm,
    minimum width=5mm,
    inner sep=0mm,
    outer sep=0mm,
    doublewidth/.style={minimum width=10mm},
    footnotesize/.style={font=\footnotesize},
    scriptsize/.style={font=\scriptsize},
  },
  tableaunode/.style={
    pretableaunode,
  },
  medtableaunode/.style={
    pretableaunode,
    font=\footnotesize,
    text height=2.75mm,
    text depth=.75mm,
    minimum height=3.5mm,
    minimum width=3.5mm
  },
  smalltableaunode/.style={
    pretableaunode,
    font=\scriptsize,
    text height=1.85mm,
    text depth=.15mm,
    minimum height=2.5mm,
    minimum width=2.5mm,
  },
  tinytableaunode/.style={
    pretableaunode,
    font=\tiny,
    text height=1.25mm,
    text depth=.15mm,
    minimum height=1.75mm,
    minimum width=1.75mm
  },
  tableaumatrix/.style={
    pretableaumatrix,
    every node/.append style={
      tableaunode,
      draw=gray,
    },
  },
  medtableaumatrix/.style={
    pretableaumatrix,
    every node/.append style={
      medtableaunode,
      draw=gray,
    },
  },
  smalltableaumatrix/.style={
    pretableaumatrix,
    every node/.append style={
      smalltableaunode,
      draw=gray,
    },
  },
  tinytableaumatrix/.style={
    pretableaumatrix,
    every node/.append style={
      tinytableaunode,
      draw=gray,
    },
  },
  tableau/.style={
    baseline=-1.25mm,
    every matrix/.style={tableaumatrix},
  },
  medtableau/.style={
    baseline=-1.25mm,
    every matrix/.style={medtableaumatrix},
  },
  smalltableau/.style={
    baseline=-1.25mm,
    every matrix/.style={smalltableaumatrix},
  },
  preshapetableaumatrix/.style={
    pretableaumatrix,
    nodes in empty cells,
    every node/.append style={
      draw=black,
      anchor=base,
      inner sep=0mm,
      outer sep=0mm,
    },
  },
  shapetableaumatrix/.style={
    preshapetableaumatrix,
    every node/.append style={
      tableaunode,
      draw=gray,
    },
  },
  medshapetableaumatrix/.style={
    preshapetableaumatrix,
    every node/.append style={
      medtableaunode,
      draw=gray,
    },
  },
  smallshapetableaumatrix/.style={
    preshapetableaumatrix,
    every node/.style={
      smalltableaunode,
      draw=gray,
    },
  },
  tinyshapetableaumatrix/.style={
    preshapetableaumatrix,
    every node/.style={
      tinytableaunode,
      draw=gray,
    },
  },
  shapetableau/.style={
    baseline=-1.25mm,
    every matrix/.style={shapetableaumatrix},
  },
  medshapetableau/.style={
    baseline=-1.25mm,
    every matrix/.style={medshapetableaumatrix},
  },
  smallshapetableau/.style={
    baseline=-1.25mm,
    every matrix/.style={smallshapetableaumatrix},
  },
  tinyshapetableau/.style={
    baseline=-1.25mm,
    every matrix/.style={tinyshapetableaumatrix},
  },
  topalign/.style={
    every matrix/.append style={name=maintableau,anchor=maintableau-1-1.base},
    baseline,
  },
  tableaulabel/.style={
    node font=\scriptsize,
    text=gray,
    execute at begin node=$\bgroup,
    execute at end node=\egroup$
  },
  tableaulabelsf/.style={
    node font=\scriptsize,
    text=gray,
    execute at begin node=$\mathsf\bgroup,
    execute at end node=\egroup$
  },
}

\newcommand*\tableau[2][]{\tikz[tableau,#1]\matrix{#2};}
\newcommand*\smalltableau[2][]{\tikz[smalltableau,#1]\matrix{#2};}

\NewDocumentCommand{\qellipsis}{ s O{} m }{%
  \IfBooleanTF{#1}{}{\unskip\space}%
  \textbracks{\ldots}#2%
}

\newcommand*{\textbracks}[1]{\textup{[}#1\textup{]}}

\newcommand*{\textangbracks}[1]{\textup{<}\kern 0.06667em\relax #1\kern 0.06667em\textup{>}}

\newcommand*\texorpdfformat[3]{%
  \texorpdfstring{#1{#3}}{#2{#3}}%
}

\ifdefined\enquote
\else
  \newcommand*{\enquote}[1]{\texorpdfstring{`#1`}{'#1'}}
\fi

\newcommand*\defterm{\texorpdfformat{\emph}{\relax}}

\NewDocumentCommand{\translitdet}{ s m }{%
  \smash{\textup{\textsuperscript{#2}}}%
  \IfBooleanF{#1}{\kern -.1em}%
}

\makeatletter
\newcommand\chyph{\penalty\@M-\hskip\z@skip}
\newcommand\cendash{\penalty\@M--\hskip\z@skip}
\newcommand\cslash{\penalty\@M/\hskip\z@skip}
\makeatother

\newcommand*\claptowidth[2]{%
    \sbox0{#2}%
    \kern 0.5\wd0\clap{#1}\kern 0.5\wd0\relax%
}
\theoremstyle{definition}
\newtheorem{definition}{Definition}[section]

\newtheorem{example}[definition]{Example}

\theoremstyle{plain}
\newtheorem{corollary}[definition]{Corollary}
\newtheorem{lemma}[definition]{Lemma}
\newtheorem{proposition}[definition]{Proposition}
\newtheorem{theorem}[definition]{Theorem}

\numberwithin{equation}{section}
\usepackage{iftex}

\DeclarePairedDelimiter{\abs}{\lvert}{\rvert}

\DeclarePairedDelimiter{\floor}{\lfloor}{\rfloor}
\DeclarePairedDelimiter{\parens}{\lparen}{\rparen}

\DeclarePairedDelimiter{\set}{\{}{\}}
\DeclarePairedDelimiterX{\gset}[2]{\{}{\}}{\,#1:#2\,}

\makeatletter
\newcommand*\powerset{\@ifstar\@powersetnoparens\@powersetparens}
\newcommand*\@powersetparens[2][]{\pset\parens[#1]{#2}}
\newcommand*\@powersetnoparens[1]{\pset{#1}}
\makeatother

\ifluatex
  \newcommand*\nset{\BbbN}

  \newcommand*\pset{\BbbP}
\else
  \newcommand*\nset{\mathbb{N}}

  \newcommand*\pset{\mathbb{P}}
\fi

\makeatletter
\newcommand*{\biggg}{\bBigg@{4}}

\newcommand*{\Biggg}{\bBigg@{5}}

\newcommand*{\sizeddelimiter}[2]{\bBigg@{#1}#2}
\makeatother

\makeatletter
\newcommand*{\sizedsurd}[2][]{%
  {\@mathmeasure\z@{\nulldelimiterspace\z@}%
     {\sqrt[#1]{\vcenter to #2\big@size{}}}%
     \box\z@}%
 }
\makeatother

\ifluatex
\makeatletter
\newcommand{\subalign}[1]{%
  \vcenter{%
    \Let@ \restore@math@cr \default@tag
    \baselineskip\Umathstacknumup\scriptstyle
    \advance\baselineskip\Umathfractiondenomdown\scriptstyle
    \lineskip=3\Umathfractiondenomvgap\scriptstyle
    \lineskiplimit\lineskip
    \ialign{\hfil$\m@th\scriptstyle##$&$\m@th\scriptstyle{}##$\hfil\crcr
      #1\crcr
    }%
  }%
}
\makeatother
\else
\makeatletter
\newcommand{\subalign}[1]{%
  \vcenter{%
    \Let@ \restore@math@cr \default@tag
    \baselineskip\fontdimen10 \scriptfont\tw@
    \advance\baselineskip\fontdimen12 \scriptfont\tw@
    \lineskip\thr@@\fontdimen8 \scriptfont\thr@@
    \lineskiplimit\lineskip
    \ialign{\hfil$\m@th\scriptstyle##$&$\m@th\scriptstyle{}##$\hfil\crcr
      #1\crcr
    }%
  }%
}
\makeatother
\fi

\newcommand*{\emptyword}{\varepsilon}

\DeclarePairedDelimiterX{\pres}[2]{\langle}{\rangle}{#1\,\delimsize\vert\,\mathopen{}#2}

\newcommand*\mathclaptowidth[2]{%
    \sbox0{$#2$}%
    \kern 0.5\wd0\mathclap{#1}\kern 0.5\wd0\relax%
}
\newcommand*{\aA}{\mathcal{A}}

\newcommand*{\plen}[2][]{\ell\parens[#1]{#2}}
\newcommand*{\pwt}[2][]{\abs[#1]{#2}}

\newcommand*{\clen}[2][]{\ell\parens[#1]{#2}}
\newcommand*{\cwt}[2][]{\abs[#1]{#2}}

\newcommand*{\evlit}{{\mathrm{ev}}}
\newcommand*{\ev}[2][]{\evlit\parens[#1]{#2}}

\newcommand*{\plac}{{\mathsf{plac}}}
\newcommand*{\hypo}{{\mathsf{hypo}}}

\newcommand*{\placcong}{\equiv_\plac}
\newcommand*{\hypocong}{\equiv_\hypo}

\newcommand*{\colreading}[2][]{\mathrm{C}\parens[#1]{#2}}

\newcommand*{\plit}{\mathrm{P}}

\newcommand*{\pplac}[2][]{\plit_{\plac}\parens[#1]{#2}}

\newcommand*{\phypo}[2][]{\plit_{\hypo}\parens[#1]{#2}}

\newcommand*\e{\ddot{e}}
\newcommand*\f{\ddot{f}}

\newcommand*\ke{\tilde{e}}
\newcommand*\kf{\tilde{f}}

\newcommand*\kecount{\tilde\epsilon}
\newcommand*\kfcount{\tilde\phi}
\tikzset{
  pickmatrix/.style={
    pretableaumatrix,
    every node/.style={
      pretableaunode,
      text=gray,
    },
    drawnode/.style={
      draw=gray,
      text=black,
    },
  },
}

\makeatletter
\DeclareRobustCommand*{\onmathaxis}[1]{%
  \ensuremath{{
    \sbox0{$\vcenter{}$}%
    \raisebox{\ht0}{$\m@th #1$}%
  }}%
}
\makeatother

\newcommand*\pickqrtlit{\mathfrak{r}}
\newcommand*\pickqrt[2]{\pickqrtlit\parens{#1,#2}}
\newcommand*\genqalit{\mathfrak{a}}
\newcommand*\genqa[1]{\genqalit\parens{#1}}

\newcommand*\setqa{\mathcal{QA}}

\newcommand*\tc{\ddot{c}}
\newcommand*\td{\ddot{d}}

\newcommand*\qrt[1]{\mathsf{QRT}\parens{#1}}
\newcommand*\yt[1]{\mathsf{YT}\parens{#1}}
\newcommand*\syt[1]{\mathsf{SYT}\parens{#1}}

\newcommand*\dc[1]{\mathsf{DesComp}\parens{#1}}

\newcommand*\Sym{\mathsf{Sym}}
\newcommand*\QSym{\mathsf{QSym}}

\newcommand*\Skel{\mathrm{Skel}}

\begin{document}

\title[Structure of quasi-crystal graphs]{Structure of quasi-crystal graphs and applications to the combinatorics of quasi-symmetric functions}

\author[A.J. Cain]{Alan J. Cain}
\address[A.J. Cain]{%
Center for Mathematics and Applications (NOVA Math)\\
NOVA School of Science and Technology\\
NOVA University of Lisbon\\
2829--516 Caparica\\
Portugal
}
\email{%
a.cain@fct.unl.pt
}
\thanks{This work is funded by national funds through the FCT – Fundação para a Ciência e a Tecnologia, I.P., under the scope of the projects UIDB/00297/2020 (https://doi.org/10.54499/UIDB/00297/2020) and UIDP/00297/2020 (https://doi.org/10.54499/UIDP/00297/2020) (Center for Mathematics and Applications)}

\author[A. Malheiro]{Ant\'onio Malheiro}
\address[A. Malheiro]{%
Center for Mathematics and Applications (NOVA Math) / Department of Mathematics \\
NOVA School of Science and Technology\\
NOVA University of Lisbon\\
2829--516 Caparica\\
Portugal
}
\email{%
ajm@fct.unl.pt
}

\author[F. Rodrigues]{F\'atima Rodrigues}
\address[F. Rodrigues]{%
Center for Mathematics and Applications (NOVA Math) / Department of Mathematics \\
NOVA School of Science and Technology\\
NOVA University of Lisbon\\
2829--516 Caparica\\
Portugal
}
\email{%
mfsr@fct.unl.pt
}

\author[I. Rodrigues]{Inês Rodrigues}
\address[I. Rodrigues]{%
Center for Mathematics and Applications (NOVA Math)\\
NOVA School of Science and Technology\\
NOVA University of Lisbon\\
2829--516 Caparica\\
Portugal
}
\email{%
ima.rodrigues@fct.unl.pt
}

\begin{abstract}
  Crystal graphs are powerful combinatorial tools for working with the plactic monoid and symmetric functions.
  Quasi-crystal graphs are an analogous concept for the hypoplactic monoid and quasi-symmetric functions. This paper
  makes a combinatorial study of these objects.
  We explain a previously-observed isomorphism of components of the quasi-crystal
  graph, and provide an explicit description using a new combinatorial structure called a quasi-array. Then two
  conjectures of Maas-Gariépy on the interaction of fundamental quasi-symmetric functions and Schur functions and on the
  arrangement of quasi-crystal components within crystal components are answered, the former positively, the latter
  negatively.
\end{abstract}

\maketitle

\section{Introduction}

Crystals, which were first introduced by Kashiwara in the early 1990s
\cite{kashiwara_oncrystalbases,kashiwara_crystalizing}, have their origin in the representation theory of quantum groups
\cite{hong_quantumgroups}, but the notion can be applied to provide a better understanding of older concepts in
combinatorics. We can identify a crystal with its crystal graph, a weighted, edge-labelled directed graph, whose edges are given by the Kashiwara operators.

As an example of crystals having applications to classical combinatorics, we highlight the case of the plactic monoid \cite{LS81}. Its elements can be viewed as semistandard Young tableaux and it is connected to the Schensted insertion algorithm: plactic classes correspond to words having the same insertion tableau, while dual equivalence classes (also called coplactic) correspond to words having the same recording tableau. The plactic monoid is associated to a crystal graph, whose connected components determine a dual equivalence class, while the isomorphisms (as weighted, labelled directed graphs) between components correspond to plactic classes. Moreover, the characters of these connected components are known to be the classical Schur functions $s_{\lambda}$.

Crystal graphs thus provide powerful combinatorial tools to work with both the classical plactic monoid and plactic
monoids corresponding to other simple Lie algebras. In particular, used in a purely combinatorial way, independently of their origin in representation theory, they allow the construction of finite complete rewriting systems and biautomatic
structures for these monoids \cite{cgm_crystal}.

This success led the first and second authors of the present paper to
seek analogous structures for other `plactic-like' monoids whose elements can be viewed as combinatorial objects of some
type, notably the hypoplactic monoid \cite{KT97, novelli_hypoplactic}, whose elements can be viewed as quasi-ribbon tableaux \cite{cm_hypoplactic}.
Quasi-ribbon tableaux have a role in the theory of quasi-symmetric functions analogous to the role of Young tableaux in the theory of symmetric functions \cite{luoto_quasisymmetric}, while also having an analogous insertion algorithm, the Krob--Thibon insertion \cite{KT97}.

This analogous `quasi-crystal graph' corresponds to the
action on words of `quasi-Kashiwara' operators, which are defined combinatorially, without reference to representations
of any underlying algebra. Analogously to the case of the plactic monoid,the connected components of the quasi-crystal graph correspond to classes of words with the same recording tableau under the Krob--Thibon insertion, and the isomorphisms (as weighted, labelled directed graphs) between components correspond to hypoplactic classes. The characters of these components are known to be fundamental quasi-symmetric functions $F_{\alpha}$ \cite{gessel_multipartite}.

Although detached completely from representation theory, quasi-crystals have turned out to be useful tools for understanding the
structure and combinatorics of the hypoplactic monoid, and the related sylvester and Baxter monoids (the monoids of
binary search trees and pairs of twin binary search trees, the latter of which is connected to the theory of Baxter
permutations) \cite{cm_sylvester,giraudo_baxter2}.

Certain components of the quasi-crystal graph for the hypoplactic monoid seem to be isomorphic
as (unweighted, unlabelled) directed graphs, as illustrated, for example in \fullref{Figure}{fig:crystal_quasicrystal} or in \cite[Figure 6]{cgm_quasicrystals}. This paper shows when these isomorphisms arise and describes them
explicitly. The existence of the isomorphisms, though not their explicit description, can be deduced from a result of
Maas-Gariépy \cite[Theorem~2]{maasgariepy_quasicrystal}, which the authors came across after starting this work.

Maas-Gariépy independently introduced an equivalent notion of quasi-crystal by restricting the action of Kashiwara
operators on semistandard Young tableaux, in order to investigate the relationship between symmetric and quasi-symmetric
functions. In contrast, the first and second authors had defined the notion in an algebraic context, in order to
understand better the hypoplactic monoid. Thus, while the notions of quasi-crystals are equivalent and there are
conceptual parallels, the two approaches barely overlap.

Maas-Gariépy's paper also contains a number of important conjectures, of which we were able to resolve two: one
positively, one negatively. The true conjecture asserts that the fundamental quasi-symmetric function $F_\alpha$ appears
in the expansion of the Schur function $s_\lambda$ in fundamental quasi-symmetric functions, where $\lambda$ is the partition obtained by re-arranging the composition $\alpha$
into non-decreasing order. This conjecture can be converted into a statement about how connected components of the
quasi-crystal graph appear in connected components of the crystal graph, which allows a proof using quasi-ribbon
tableaux and Young tableaux, and their readings.

The negative conjecture concerns the `skeleton' of the connected components of the crystal graph, which, loosely speaking, is a graph describing how connected components of the quasi-crystal graph are situated within a connected component of the crystal graph. We prove that part of this conjecture is true, and provide an explicit example disproving the remaining part.

This paper is organized as follows. \fullref{Section}{sec:preliminaries} recalls the main notions concerning the plactic and hypoplactic monoids, and their connection to crystals and quasi-crystals, respectively. Likewise, we recall the notion of Schur functions and fundamental quasi-symmetric functions.
In \fullref{Section}{sec:quasiarrays}, we introduce the notion of quasi-array and certain operators, that will be used to introduce the quasi-array graph, that will be used on the next section. In \fullref{Section}{sec:connection_qc}, we show explicit isomorphisms between certain quasi-array graphs and quasi-crystal graphs (\fullref{Theorems}{thm:isom-quasi-array-graph-quasi-crystal}, for the finite rank case, \and \fullref{Theorem}{thm:isom-quasi-array-graph-quasi-crystal-finite-rank} for finite rank), which implies the isomorphism (as unlabelled graphs) of certain quasi-crystal components (\fullref{Corollaries}{corol:isomorphism-unlabelled} and \ref{corol:isom-compositions-finite-rank}). Following these isomorphisms, we explore in \fullref{Section}{sec:geometry} a natural geometric interpretation of connected components of quasi-crystals as lattices. In \fullref{Section}{sec:schur_fund}, we show some application on the relationship of Schur functions and fundamental quasi-symmetric functions, which proves a conjecture of Maas-Gariépy. In \fullref{Section}{sec:skeleton}, we recall the notion of skeleton of a connected quasi-crystal graph, then prove part of a conjecture of Maas-Gariépy (\fullref{Proposition}{prop:even-len}), and disprove the remaining of this conjecture, by exhibiting a counter-example (\fullref{Example}{ex:counter-example}).

\section{Preliminaries}\label{sec:preliminaries}

\subsection{Alphabets}

Throughout this paper, $\aA$ will be the set of natural numbers viewed as an infinite ordered alphabet:
$\aA = \set{1 < 2 < 3 < \ldots}$. Further, $n$ will be a natural number and $\aA_n$ will be the set of the first $n$
natural numbers viewed as an ordered alphabet: $\aA_n = \set{1 < 2 < \ldots < n}$. The free monoid on an alphabet $X$
(that is, the set of words over $X$, including the empty word $\emptyword$, under the operation of concatenation) is
denoted $X^*$.

For any $w \in \aA^*$ and $a \in \aA$, $\abs{w}_a$ is the number of symbols $a$ contained in $w$. For any $w \in \aA^*$,
$\ev{w}$ is the tuple whose $a$-th component is $|w|_a$; such tuples are truncated to finite length by ignoring the
infinite suffix of entries $0$.

\subsection{Arrays}

This subsection introduces notation and concepts for arrays of cells, each of which may contain an entry from $\aA$.
Young tableaux and quasi-ribbon tableaux are particular kinds of array. Rows of arrays are indexed from top to bottom;
columns are indexed from left to right. Thus the $(i,j)$-th cell of an array is the cell in the $i$-th row from the top
and the $j$-th column from the left. When the $(i,j)$-th cell of an array $A$ contains an entry from $\aA$, this element
of $\aA$ is denoted $A_{(i,j)}$. Further, $\ev{A}$ is the tuple whose $a$-th component is the number of symbols
$a \in \aA$ appearing in $A$; as before, such tuples are truncated to finite length by ignoring the infinite suffix of
entries $0$.

For $k = 1,\ldots,m$, the \defterm{$k$-th diagonal} of an array comprises the $(i,j)$-th cells where $i + j - 1 = k$.
That is, the $k$-th diagonal comprises whichever of the $(1,k)$, $(2,k-1)$, $(3,k-2)$, \ldots, $(k-1,2)$, $(k,1)$-th
cells lie in the array. The following diagrams illustrate the $k$-th diagonals of a Young tableau and a quasi-ribbon
tableau for $k \in \set{1,2,\ldots,5}$:
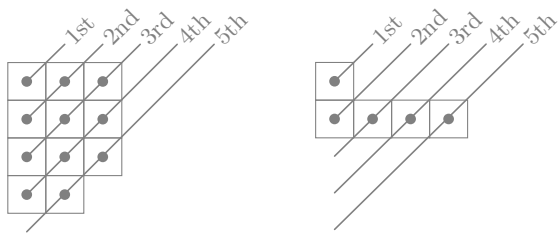
\[
  \begin{tikzpicture}[baseline=(maintableau-1-1.base)]
    \matrix[pickmatrix] (maintableau) {
      \null \& \null \& \null \& \null \& \null \& \null \\
      |[drawnode]| \null \& |[drawnode]| \null \& |[drawnode]| \null \& \null \& \null \\
      |[drawnode]| \null \& |[drawnode]| \null \& |[drawnode]| \null \& \null \\
      |[drawnode]| \null \& |[drawnode]| \null \& |[drawnode]| \null \\
      |[drawnode]| \null \& |[drawnode]| \null \\
      \null \\
    };
    \foreach \j in {2,3,4} {
      \foreach \i in {1,2,3} {
        \fill[gray] (maintableau-\j-\i.center) circle (2pt);
      }
    }
    \fill[gray] (maintableau-5-1.center) circle (2pt);
    \fill[gray] (maintableau-5-2.center) circle (2pt);
    \begin{scope}[
      draw=gray,
      semithick,
      line cap=round,
      every node/.style={
        text=gray,
        inner sep=2pt,
        outer sep=0,
        node font=\footnotesize,
      }
      ]
      \draw (maintableau-1-2.center) -- node[anchor=west,pos=0,rotate=45] {$1$st} (maintableau-2-1.center);
      \draw (maintableau-1-3.center) -- node[anchor=west,pos=0,rotate=45] {$2$nd} (maintableau-3-1.center);
      \draw (maintableau-1-4.center) -- node[anchor=west,pos=0,rotate=45] {$3$rd} (maintableau-4-1.center);
      \draw (maintableau-1-5.center) -- node[anchor=west,pos=0,rotate=45] {$4$th} (maintableau-5-1.center);
      \draw (maintableau-1-6.center) -- node[anchor=west,pos=0,rotate=45] {$5$th} (maintableau-6-1.center);
    \end{scope}
  \end{tikzpicture}
  \qquad
  \begin{tikzpicture}[baseline=(maintableau-1-1.base)]
    \matrix[pickmatrix] (maintableau) {
      \null \& \null \& \null \& \null \& \null \& \null \\
      |[drawnode]| \null \& \null \& \null \& \null \& \null \\
      |[drawnode]| \null \& |[drawnode]| \null \& |[drawnode]| \null \& |[drawnode]| \null \\
      \null \& \null \& \null \\
      \null \& \null \\
      \null \\
    };
    \foreach \i in {1,2,3,4} {
      \fill[gray] (maintableau-3-\i.center) circle (2pt);
    }
    \fill[gray] (maintableau-2-1.center) circle (2pt);
    \begin{scope}[
      draw=gray,
      semithick,
      line cap=round,
      every node/.style={
        text=gray,
        inner sep=2pt,
        outer sep=0,
        node font=\footnotesize,
      }
      ]
      \draw (maintableau-1-2.center) -- node[anchor=west,pos=0,rotate=45] {$1$st} (maintableau-2-1.center);
      \draw (maintableau-1-3.center) -- node[anchor=west,pos=0,rotate=45] {$2$nd} (maintableau-3-1.center);
      \draw (maintableau-1-4.center) -- node[anchor=west,pos=0,rotate=45] {$3$rd} (maintableau-4-1.center);
      \draw (maintableau-1-5.center) -- node[anchor=west,pos=0,rotate=45] {$4$th} (maintableau-5-1.center);
      \draw (maintableau-1-6.center) -- node[anchor=west,pos=0,rotate=45] {$5$th} (maintableau-6-1.center);
    \end{scope}
  \end{tikzpicture}
\]

\subsection{Young tableau and the plactic monoid}

Let $\lambda = \parens{\lambda_1,\ldots,\lambda_{\plen\lambda}}$ be a partition.

\begin{definition}
A \defterm{Young diagram} of shape $\lambda$ is an array of cells, with $\lambda_i$ cells in the $i$-th row,
left-aligned.
\end{definition}

\begin{definition}
A \defterm{Young tableau} of shape $\lambda$ is a filling of a Young diagram with symbols from $\aA$ such that the entries in each row are non-decreasing
from left to right, and the entries in each column are strictly increasing from top to bottom. The set of all Young tableaux of shape $\lambda$ is denoted $\yt{\lambda}$.
\end{definition}

 An example of a Young
tableau of shape $(5,3,2,2)$ is
\begin{equation}
\label{eq:youngtableaueg}
\tableau{
1 \& 1 \& 3 \& 3 \& 6 \\
2 \& 3 \& 4 \\
4 \& 4 \\
5 \& 6 \\
}.
\end{equation}

\begin{definition}
A \defterm{standard Young tableau} of shape $\lambda$ is a Young tableau that is filled with symbols from
$\set{1,\ldots,\pwt\lambda}$, with each symbol appearing exactly once. The set of all standard Young tableaux of shape $\lambda$ is denoted $\syt{\lambda}$.	
\end{definition}

 Given a Young tableau $T$, its \emph{standardization} is the standard Young tableau obtained by replacing the letters $i$ consecutively from left to right with
\begin{align*}
	\alpha_1 &+ \alpha_2 + \ldots + \alpha_{i-1}+1,\\
	\alpha_1 &+ \alpha_2 + \ldots + \alpha_{i-1}+2,\\
	&\ldots\\
	\alpha_1 &+ \alpha_2 + \ldots + \alpha_{i-1}+ \alpha_i,
\end{align*}
for $1 \leq i \leq k$, with $\alpha_i$ being the number of occurrences of the letter $i$ in $T$.

For example, the following is a standard tableau, which is the standardization of the tableau \eqref{eq:youngtableaueg}:
\begin{equation}
	\label{eq:syteg}
	\tableau{
		1 \& 2 \& 5 \& 6 \& 12 \\
		3 \& 4 \& 9 \\
		7 \& 8 \\
		10 \& 11 \\
	}.
\end{equation}

The \defterm{column reading} $\colreading{T}$ of a tableau $T$ is the word in $\aA^*$ obtained by proceeding through the
columns, from leftmost to rightmost, and reading each column from bottom to top. For example, the column reading of the
tableau \eqref{eq:youngtableaueg} is $5421\,6431\,41\,3\,6$.

For any word $w \in \aA^*$, Schensted's algorithm computes a Young tableau $\pplac{w}$; for the details of this
algorithm, see \cite[\S~1.1]{fulton_young} or \cite[\S~3.1]{cm_hypoplactic}. The essential fact here is that
$\pplac{\colreading{T}} = T$ for any Young tableau $T$.

The \defterm{plactic congruence} $\placcong$ is defined on $\aA^*$ by the following:
\[
u \placcong v \iff \pplac{u} = \pplac{v}.
\]

\begin{example}\label{ex:placcong}
	We have $2113 \placcong 1213$, since
	$$\pplac{2113} = \pplac{1213} = \tableau{
		1 \& 1 \& 3\\
		2\\}.$$
\end{example}

\begin{definition}
	The \defterm{plactic monoid (of infinite rank)} $\plac$ is the factor monoid $\aA^*/{\placcong}$. The congruence $\placcong$
	naturally restricts to a congruence on $\aA_n^*$, and the factor monoid $\aA_n^*/{\placcong}$ is the \defterm{plactic monoid of rank $n$} (or finite rank) and is denoted $\plac_n$.
\end{definition}

\subsection{Quasi-ribbon tableaux and the hypoplactic monoid}

This subsection briefly recalls some essential definitions; see \cite[especially\ \S~4]{cm_hypoplactic} for more
background.

\begin{definition}
A \defterm{quasi-ribbon diagram} of shape $\sigma$, where $\sigma = (\sigma_1,\ldots,\sigma_k)$ is a composition, is an
array of cells, with $\sigma_i$ cells in the $i$-th row (counting from the top), aligned so that the leftmost cell of
the $i+1$-th row is below the rightmost cell of the $i$-th.	
\end{definition}

\begin{definition}

Given a composition $\sigma$, a \defterm{quasi-ribbon tableau} of shape $\sigma$ is a filling of a quasi-ribbon diagram of the same shape, with symbols from $\aA$ such that the entries in every row are non-decreasing from left to right and the entries in every column are strictly increasing from top to
bottom.
\end{definition}

 An example of a quasi-ribbon tableau is
\begin{equation}
\label{eq:qrteg}
\tableau{
 1 \& 2 \& 2                          \\
   \&   \& 3                          \\
   \&   \& 4 \& 4 \& 5 \& 5 \& 5      \\
   \&   \&   \&   \&   \&   \& 6 \& 7 \\
}.
\end{equation}
Note the following immediate consequences of the definition of a quasi-ribbon tableau: (1) for each $a \in \aA$, the
symbols $a$ in a quasi-ribbon tableau all appear in the same row, which must be the $j$-th row for some $j \leq a$; (2)
the $h$-th row of a quasi-ribbon tableau cannot contain symbols from $\set{1,\ldots,h-1}$.

The set of all quasi-ribbon tableaux of shape $\sigma$ is denoted $\qrt{\sigma}$.

The \defterm{column reading} $\colreading{T}$ of a quasi-ribbon tableau $T$ is the word in $\aA^*$ obtained by
proceeding through the columns, from leftmost to rightmost, and reading each column from bottom to top. For example, the
column reading of the quasi-ribbon tableau \eqref{eq:qrteg} is $1\,2\,432\,4\,5\,5\,65\,7$.

Given a word $u \in \aA^*$, there is an insertion algorithm that computes a quasi-ribbon tableau $\phypo{u}$ from $u$.
For the details of this algorithm, see \cite[Algorithm~4.4]{novelli_hypoplactic} or \cite[\S~4]{cm_hypoplactic}. If the
symbols appearing in $u$ are $\set{a_1 < a_2 < \ldots < a_m}$ (where $a_i \in \aA$), then $\phypo{u}$ is the unique
quasi-ribbon tableau containing $|u|_{a_i}$ entries $a_i$, with the entries $a_i$ and $a_{i+1}$ on different rows if and
only if $u$ contains a subsequence $a_{i+1}a_i$. A straightforward consequence of this is that
$\phypo{\colreading{T}} = T$ for any quasi-ribbon tableau $T$.

The \defterm{hypoplactic congruence} is defined on $\aA^*$ by
\[
  u \hypocong v \iff \phypo{u} = \phypo{v}.
\]

\begin{example}\label{ex:hypocong}
	The words $2131$ and $1213$ are hypoplactic congruent, since
	$$\phypo{2131} = \phypo{1213} = \tableau{1 \& 1\\
											   \& 2 \& 3\\}.$$

	We remark that these words are not plactic congruent, as we have
	$$\pplac{2131} = \tableau{1 \& 1\\
		2 \& 3\\} \neq
	\tableau{1 \& 1 \& 3\\
		2\\} = \pplac{1213}.$$
\end{example}

\begin{definition}
	The \defterm{hypoplactic monoid (of infinite rank)} $\hypo$ is the factor monoid $\aA^*/{\hypocong}$. The congruence
	$\hypocong$ naturally restricts to a congruence on $\aA_n^*$, and the factor monoid $\aA_n^*/{\hypocong}$ is the
	\defterm{hypoplactic monoid of rank $n$} (or finite rank) and is denoted $\hypo_n$.
\end{definition}

\subsection{Crystal and quasi-crystal graphs}

This subsection recalls the basic definitions and key results regarding the crystal and quasi-crystal graph and their
interaction; see \cite{cm_hypoplactic} for a full treatment.

\subsubsection{Crystals}

\begin{definition}
Given $i \in \nset$, the partial \defterm{Kashiwara operators} $\ke_i$ and $\kf_i$ are defined on $u \in \aA_n^*$ as follows.
Form a new word in $\set{{+},{-}}^*$ by replacing each letter $i$ of $u$ by the symbol $+$, each letter $i+1$ by the
symbol $-$, and every other symbol with the empty word, keeping a record of the original letter replaced by each symbol.
Then delete factors ${-}{+}$ until no such factors remain: the resulting word is
$\rho_i(uw) = {+}^{\kfcount_i(uw)}{-}^{\kecount_i(u)}$. Note that factors ${+}{-}$ are \emph{not} deleted.
\end{definition}

Regarding the operator $\ke_i$:
\begin{itemize}
  \item If $\kecount_i(u)=0$, then $\ke_i(u)$ is undefined.
  \item If $\kecount_i(u)>0$ then one obtains $\ke_i(u)$ by taking the letter $i+1$ which was replaced by the leftmost
        $-$ of $\rho_i(u)$ and changing it to $i$.
\end{itemize}
Regarding the operator $\kf_i$:
\begin{itemize}
  \item If $\kfcount_i(u)=0$, then $\kf_i(u)$ is undefined.
  \item If $\kfcount_i(u)>0$ then one obtains $\kf_i(u)$ by taking the letter $i$ which was replaced by the rightmost
        $+$ of $\rho_i(u)$ and changing it to $i+1$.
\end{itemize}
The operators $\ke_i$ and $\kf_i$ are mutually inverse, in the sense that if $\ke_i(u)$ is defined,
$u = \kf_i(\ke_i(u))$, and if $\kf_i(u)$ is defined, $u = \ke_i(\kf_i(u))$.

\begin{definition}
The \defterm{crystal graph} for $\plac$, denoted $\Gamma(\plac)$, is the directed labelled graph with vertex set $\aA^*$
and, for $u,v \in \aA^*$, an edge from $u$ to $v$ labelled by $i$ if and only if $v = \kf_i(u)$ (or, equivalently,
$u = \ke_i(v)$). The crystal graph for $\plac_n$, denoted $\Gamma(\plac_n)$, is the subgraph of $\Gamma(\plac)$ induced
by $\aA_n^*$.
\end{definition}

Notice that edge labels in $\Gamma(\plac_n)$ must lie in $\set{1,\ldots,n-1}$. For any $w \in \aA^*$, let
$\Gamma(\plac,w)$ (respectively, $\Gamma(\plac_n,w)$) denote the connected component of $\Gamma(\plac)$ (respectively,
$\Gamma(\plac_n)$) that contains the vertex $w$.

The fundamental connection between crystals and the plactic monoid is the following result:

\begin{theorem}[\cite{KN94}]
Given words $u, v \in \aA^*$, $u \placcong v$ if and only
if there is a weighted labelled digraph isomorphism $\theta : \Gamma(\plac,u) \to \Gamma(\plac,v)$ or
$\theta : \Gamma(\plac_n,u) \to \Gamma(\plac_n,v)$ with $\theta(u) = v$.
\end{theorem}

Here, \defterm{weighted} means that $\ev{w} = \ev{\theta(w)}$ for all $w \in \Gamma(\plac,u)$ or
$w \in \Gamma(\plac_n,u)$. In fact, in the infinite-rank case (that is, $\plac$), all labelled digraph isomorphisms
between connected components of $\Gamma(\plac)$ are weighted; this does not hold in general for $\Gamma(\plac_n)$.
Henceforth, isomorphisms between connected components of $\Gamma(\plac_n)$ are assumed to be weighted, and in this
context the term `isomorphic' implies `weighted'.

The column readings of Young tableaux (respectively, of Young tableaux with entries from $\aA_n$) form a union of
connected components of $\Gamma(\plac)$ (respectively $\Gamma(\plac_n)$). Furthermore, each of these connected
components consists precisely of the column readings of the Young tableaux (respectively, of the Young tableaux with
entries from $\aA_n$) of a particular shape. The connected component comprising the column readings of the Young
tableaux (respectively, of Young tableaux with entries from $\aA_n$) of shape $\lambda$ is denoted
$\Gamma(\plac,\lambda)$ (respectively, $\Gamma(\plac_n,\lambda)$) and each vertex $u$ of this component can be
identified with the Young tableau $\pplac{u}$.

\subsubsection{Quasi-crystals}

\begin{definition}
Given $i \in \nset$, the partial \defterm{quasi-Kashiwara operators} $\e_i$ and $\f_i$ are defined on $u \in\aA^*$ as follows.
\begin{itemize}
\item If $u$ contains a subsequence $(i+1)i$, both $\e_i(u)$ and $\f_i(u)$ are undefined.
\item If $u$ does not contain a subsequence $(i+1)i$, but $u$ contains at least one symbol $i+1$, then $\e_i(u)$ is the
  word obtained from $u$ by replacing the left-most symbol $i+1$ by $i$; if $u$ contains no symbol $i+1$, then $\e_i(u)$
  is undefined.
\item If $u$ does not contain a subsequence $(i+1)i$, but $u$ contains at least one symbol $i$, then $\f_i(u)$ is the
  word obtained from $u$ by replacing the right-most symbol $i$ by $i+1$; if $u$ contains no symbol $i$, then $\f_i(u)$
  is undefined.
\end{itemize}
\end{definition}

The operators $\e_i$ and $\f_i$ are mutually inverse, in the sense that if $\e_i(u)$ is defined, $u = \f_i(\e_i(u))$,
and if $\f_i(u)$ is defined, $u = \e_i(\f_i(u))$.

\begin{example}\label{ex:quasi_kashiwara}
	Let $u = 12211 \in \aA_3^*$. Since $u$ contains a subsequence $21$, the quasi-Kahsiwara operators $\e_1$ and $\f_1$ are undefined. Moreover, $u$ does not contain a subsequence $32$. Thus, as it does not contain any symbol $3$, $\e_2$ is undefined, and since it contains at least one symbol $2$, $\f_2$ is defined and obtained by changing the right-most symbol $2$ to $3$, that is
	$$\f_2 (u) = 12311.$$
\end{example}

\begin{definition}
The \defterm{quasi-crystal graph} for $\hypo$, denoted $\Gamma(\hypo)$, is the directed labelled graph with vertex set
$\aA^*$ and, for $u,v \in \aA^*$, an edge from $u$ to $v$ labelled by $i$ if and only if $v = \f_i(u)$ (or,
equivalently, $u = \e_i(v)$). The quasi-crystal graph for $\hypo_n$, denoted $\Gamma(\hypo_n)$, is the subgraph of
$\Gamma(\hypo)$ induced by $\aA_n^*$.
\end{definition}

 Notice that edge labels in $\Gamma(\hypo_n)$ must lie in $\set{1,\ldots,n-1}$. For
any $w \in \aA^*$, let $\Gamma(\hypo,w)$ (respectively, $\Gamma(\hypo_n,w)$) denote the connected component of
$\Gamma(\hypo)$ (respectively, $\Gamma(\hypo_n)$) that contains the vertex $w$.

The fundamental connection between quasi-crystals and the hypoplactic monoid is the following result:

\begin{theorem}[{\cite[Theorem 1]{cm_hypoplactic}}]
	Given words $u,v \in \aA^*$, $u \hypocong v$ if
	and only if there is a weighted labelled digraph isomorphism $\theta : \Gamma(\hypo,u) \to \Gamma(\hypo,v)$ or
	$\theta : \Gamma(\hypo_n,u) \to \Gamma(\hypo_n,v)$ with $\theta(u) = v$.
\end{theorem}
Just as for the $\plac$, all labelled digraph
isomorphisms between connected components of $\Gamma(\hypo)$ are weighted; this does not hold in general for
$\Gamma(\hypo_n)$. Henceforth, isomorphisms between connected components of $\Gamma(\hypo_n)$ are assumed to be
weighted, and in this context the term `isomorphic' implies `weighted'.

\begin{figure}
	\begin{tikzpicture}[yscale=1.2]
		\node (1) at (2,4) {211};
		\node (a) at (0,3) {212};
		\node (2) at (3,3) {311};
		\node (b) at (0,2) {213};
		\node (3) at (3,2) {312};
		\node (c) at (0,1) {313};
		\node (4) at (3,1) {322};
		\node (d) at (1,0) {323};
		\draw (1) edge[->, very thick, blue, above right] node[black] {\footnotesize $2$} (2);
		\draw (2) edge[->, very thick, red, right] node[black] {\footnotesize $1$} (3);
		\draw (3) edge[->, very thick, red, right] node[black] {\footnotesize $1$} (4);

		\draw (a) edge[->, very thick, blue, left] node[black] {\footnotesize $2$} (b);
		\draw (b) edge[->, very thick, blue, left] node[black] {\footnotesize $2$} (c);
		\draw (c) edge[->, very thick, red, below left] node[black] {\footnotesize $1$} (d);

		\draw (1) edge[->, dashed, red, above left] node[black] {\footnotesize $1$} (a);
		\draw (4) edge[->, dashed, blue, below right] node[black] {\footnotesize $2$} (d);
	\end{tikzpicture}
\caption{The connected components $\Gamma (\hypo_3, 211)$ and $\Gamma(\hypo_3, 212)$ (solid edges). Adding the dashed edges, one obtains the connected component $\Gamma (\plac_3, 211)$.}
\label{fig:crystal_quasicrystal}
\end{figure}

\begin{figure}
	\begin{tikzpicture}[yscale=1.2]
	\node (1) at (3,4) {211};

	\node (a) at (0,3) {212};
	\node (2) at (4,3) {311};

	\node (b) at (0,2) {213};
	\node (3) at (4,2) {312};
	\node (aa) at (5,2) {411};

	\node (c) at (0,1) {313};
	\node (cc) at (1,1) {214};
	\node (4) at (4,1) {322};
	\node (4a) at (5,1) {412};
	\node (4b) at (6,1) {511};

	\node (d) at (1,0) { };
	\node (d1) at (0,0) {};
	\node (d2) at (2,0) {};

	\node (f1) at (4,0) {};
	\node (f2) at (5,0) {};
	\node (f3) at (6,0) {};
	\node (f4) at (7,0) {};

	\draw (1) edge[->, very thick, blue, above right] node[black] {\footnotesize $2$} (2);
	\draw (2) edge[->, very thick, red, left] node[black] {\footnotesize $1$} (3);
	\draw (3) edge[->, very thick, red, left] node[black] {\footnotesize $1$} (4);

	\draw (a) edge[->, very thick, blue, left] node[black] {\footnotesize $2$} (b);
	\draw (b) edge[->, very thick, blue, left] node[black] {\footnotesize $2$} (c);
	\draw (c) edge[-, very thick, red, right] node[black] {\footnotesize $1$} (d);

	\draw (2) edge[->, very thick, ForestGreen, above] node[black] {\footnotesize $3$} (aa);
	\draw (b) edge[->, very thick, ForestGreen, above] node[black] {\footnotesize $3$} (cc);
	\draw (c) edge[-, very thick, ForestGreen, left] node[black] {\footnotesize $3$} (d1);
	\draw (cc) edge[-, very thick, blue, left] node[black] {\footnotesize $2$} (d1);
	\draw (cc) edge[-, very thick, brown, above] node[black] {\footnotesize $4$} (d2);

	\draw (3) edge[->, very thick, ForestGreen, above] node[black] {\footnotesize $3$} (4a);
	\draw (aa) edge[->, very thick, red, right] node[black] {\footnotesize $1$} (4a);
	\draw (aa) edge[->, very thick, brown, above] node[black] {\footnotesize $4$} (4b);

	\draw (1) edge[->, dashed, red, above left] node[black] {\footnotesize $1$} (a);

	\draw (4) edge[-, dashed, blue, above] node[black] {\footnotesize $2$} (d);

	\draw (4) edge[-, very thick, ForestGreen, left] node[black] {\footnotesize $3$} (f1);
	\draw (4a) edge[-, very thick, red, left] node[black] {\footnotesize $1$} (f1);
	\draw (4b) edge[-, very thick, red, right] node[black] {\footnotesize $1$} (f3);
	\draw (4a) edge[-, very thick, blue, left] node[black] {\footnotesize $2$} (f2);
	\draw (4a) edge[-, very thick, brown, above] node[black] {\footnotesize $4$} (f3);
	\draw (4b) edge[-, very thick, cyan, above] node[black] {\footnotesize $5$} (f4);

	\end{tikzpicture}
	\caption{Part of the connected components $\Gamma (\hypo, 211)$ and $\Gamma (\hypo, 212)$, in infinite rank (solid edges). Adding the dashed edges, one obtains part of the connected component $\Gamma(\plac, 211)$. }
	\label{fig:quasicrystal_infinite}
\end{figure}

The column readings of quasi-ribbon tableaux (respectively, of quasi-ribbon tableaux with entries from $\aA_n$) form a
union of connected components of $\Gamma(\hypo)$ (respectively $\Gamma(\hypo_n)$). Furthermore, each of these connected
components consists precisely of the column readings of the quasi-ribbon tableaux (respectively, of quasi-ribbon
tableaux with entries from $\aA_n$) of a particular shape. The connected component comprising the column readings of the
quasi-ribbon tableaux (respectively, of quasi-ribbon tableaux with entries from $\aA_n$) of shape $\sigma$ is denoted
$\Gamma(\hypo,\sigma)$ (respectively, $\Gamma(\hypo_n,\sigma)$) and each vertex $u$ of this component can be
identified with the quasi-ribbon tableau $\phypo{u}$.

Throughout this paper, we will use the terminology `infinite rank' (respectively `finite rank') whenever we refer to the monoids $\plac$ or $\hypo$ (respectively $\plac_n$ and $\hypo_n$) and their related crystal and quasi-crystal graphs. To illustrate these differences, we refer to \fullref{Figures}{fig:crystal_quasicrystal} and \ref{fig:quasicrystal_infinite}.

\subsubsection{Relationship between crystal and quasi-crystal graphs}

Whenever $\e_i$ is defined, so is $\ke_i$, and similarly whenever $\f_i$ is defined, so is $\kf_i$. Thus every action of
a quasi-Kashiwara operator is also an action of a Kashiwara operator. Thus the quasi-crystal graph $\Gamma(\hypo)$
(respectively, $\Gamma(\hypo_n)$) is a subgraph of the crystal graph $\Gamma(\plac)$ (respectively, $\Gamma(\plac_n)$).
Therefore each connected component of the crystal graph $\Gamma(\plac)$ (respectively, $\Gamma(\plac_n)$) is a union of
connected components of $\Gamma(\hypo)$ (respectively, $\Gamma(\hypo_n)$).

Using slightly inconsistent language, an action of a \defterm{strict} Kashiwara operator is an action of a Kashiwara
operator that is not the action of a quasi-Kashiwara operator. The actions of strict Kashiwara and quasi-Kashiwara
operators are distinguished by their effects on the minimal parsings of Young tableaux. Only the definition and
essential details of minimal parsings are recalled below; see \cite[\S~1.2.2]{maasgariepy_quasicrystal} for
further background.

\begin{definition}\label{def:min_pars}
	The \defterm{minimal parsing} of a Young tableau $T$ is the division of $T$ into a minimum number of bands of cells. A
\defterm{band} is a set of cells, at most one from each column, such that
\begin{enumerate}
  \item if it contains cells $(i,j)$ and $(i',j')$ with $j < j'$, then $i \geq i'$ and $T_{(i,j)} \leq T_{(i',j')}$;
  \item if it contains the cell with an entry $a$, it contains all cells with entries $a$ in $T$.
\end{enumerate}
A minimal parsing is said to have \defterm{type} $\alpha$ if its $i$-th band has length $\alpha_i$.
\end{definition}

\begin{example}\label{ex:min_pars}
For instance, the minimal parsing of the Young tableau \eqref{eq:youngtableaueg} is
\[
  \begin{tikzpicture}
    \matrix[pickmatrix] (maintableau) {
      |[drawnode]| 1 \& |[drawnode]| 1 \&  3 \& 3 \& 6 \\
      2 \& 3 \& 4 \\
      4 \& 4 \\
      5 \& 6 \\
    };
  \end{tikzpicture}
  \quad
  \begin{tikzpicture}
    \matrix[pickmatrix] (maintableau) {
      1 \& 1 \& |[drawnode]| 3 \& |[drawnode]| 3 \& 6 \\
      |[drawnode]| 2 \& |[drawnode]| 3 \& 4 \\
      4 \& 4 \\
      5 \& 6 \\
    };
  \end{tikzpicture}
  \quad
  \begin{tikzpicture}
    \matrix[pickmatrix] (maintableau) {
      1 \& 1 \& 3 \& 3 \& 6 \\
      2 \& 3 \& |[drawnode]| 4 \\
      |[drawnode]| 4 \& |[drawnode]| 4 \\
      5 \& 6 \\
    };
  \end{tikzpicture}
  \quad
  \begin{tikzpicture}
    \matrix[pickmatrix] (maintableau) {
      1 \& 1 \& 3 \& 3 \& |[drawnode]| 6 \\
      2 \& 3 \& 4 \\
      4 \& 4 \\
      |[drawnode]| 5 \& |[drawnode]| 6 \\
    };
  \end{tikzpicture}
\]
and it has type $(2,4,3,3)$.
\end{example}

Note that the minimal parsing is the pattern of \emph{cells}, not of entries.

Actions of quasi-Kashiwara operators \emph{do not} alter the minimal parsing of the Young tableau. Actions of strict
Kashiwara operators \emph{do} alter its minimal parsing.

For any Young tableau, there is a unique standard Young tableau with the same minimal parsing. For example,
\eqref{eq:syteg} is the unique standard Young tableau with the same minimal parsing as \eqref{eq:youngtableaueg}.

\begin{definition}
The \defterm{descent set} of a standard Young tableau is the set of entries $i$ such that $i + 1$ appears in a row of
greater index; each such $i$ is a \defterm{descent}. The \defterm{descent composition} associated to the descent set
$\set{i_1 < i_2 < \ldots < i_k}$ of a standard Young tableau of size $n$ is
$$\alpha = \parens{i_1, i_2 - i_1, i_3 - i_2, \ldots, n - i_k}.$$
The descent composition of a standard Young tableau $T$ is denoted $\dc{T}$.
\end{definition}

For standard tableaux, notice that the descent composition gives the lengths of the bands making up the minimal parsing, and thus coincides with its type.

\begin{example}\label{ex:descent}
Considering the standard Young tableau \eqref{eq:syteg}, its descent set is $\{2,6,9\}$ and its descent composition is $(2,4,3,3)$. The descent composition coincides with the minimal parsing type of the Young tableau \eqref{eq:youngtableaueg}, whose standardization is \eqref{eq:syteg}.
\end{example}

We remark that the descent composition does not uniquely determine a minimal parsing, and thus, it is possible for distinct quasi-crystal graphs to be indexed by the same descent composition, as illustrated by the following example.

\begin{example}
The following tableaux have distinct minimal parsings, with the same descent composition $(2,3,1)$:
$$\tableau{
	1 \& 1 \& 2\\
	2 \& 2\\
	3\\} \qquad
\tableau{
	1 \& 1 \& 2 \& 2\\
	2\\
	3\\}$$
\end{example}

Let $T$ be a Young tableau and let $\sigma$ be the descent composition of its standardization, the unique standard Young tableau with the same minimal parsing as $T$. Then $\sigma$ is the shape of $\phypo{u}$, for any word $u$ with $\pplac{u} = T$. Furthermore,
$\Gamma(\hypo,u)$ is made up of precisely those words $w \in \Gamma(\plac,u)$ such that $\pplac{w}$ has the same minimal
parsing as $T$.

\subsection{Symmetric and quasi-symmetric functions}

The essential facts about the rings of symmetric and quasi-symmetric functions are recalled here; for background, see
\cite{gessel_multipartite}.

The ring of quasi-symmetric functions $\QSym$ has as a basis the fundamental quasi-symmetric functions.

\begin{definition}\label{def:fundamental_function}
Given a composition $\alpha$, the \defterm{fundamental quasi-symmetric function} $F_\alpha$ is defined
by
\[
  F_\alpha = \sum_{\alpha\preceq\beta} M_\beta
\]
where $\alpha \preceq \beta$ indicates that the composition $\beta = \parens{\beta_1,\ldots,\beta_{\clen{\beta}}}$ is a
refinement of $\alpha$, and where $M_\beta$ is the sum of monomials
$x_{i_1}^{\beta_1}\cdots x_{i_{\clen{\beta}}}^{\beta_{\clen{\beta}}}$ where $i_1 < i_2 < \ldots < i_{\clen{\beta}}$.
\end{definition}

Each such monomial corresponds to a quasi-ribbon tableau of shape $\alpha$ containing $\beta_j$ symbols $i_j \in \aA$.
Hence
\[
  F_\alpha = \sum_{Q \in \qrt{\alpha}} x^Q,
\]
where $x^Q$ denotes the monomial $x_1^{\alpha_1}\cdots x_n^{\alpha_n}$ with
$\ev{Q} = \parens{\alpha_1,\ldots,\alpha_n}$. Thus, each $F_\alpha$ corresponds to the isomorphism class of connected
components of $\Gamma(\hypo)$ where the quasi-ribbon tableaux have shape $\alpha$ \cite{cm_hypoplactic, maasgariepy_quasicrystal}.

The ring of symmetric functions $\Sym$ has as basis the \defterm{Schur functions}. For a partition $\lambda$, the Schur
function $s_\lambda$ is defined by
\[
  s_\lambda = \sum_{T \in \yt{\lambda}} x^T
\]
where $x^T$ denotes the monomial $x_1^{\alpha_1}\cdots x_n^{\alpha_n}$ where
$\ev{T} = \parens{\alpha_1,\ldots,\alpha_n}$ (here $n$ is the largest symbol appearing in $T$). Thus each $s_\lambda$
corresponds to the isomorphism class of connected components of $\Gamma(\plac)$ where the Young tableaux have shape
$\lambda$.

Combining these observations, one obtains the following result due to Gessel \cite{gessel_multipartite},
\[
  s_\lambda = \sum_{T \in \syt{\lambda}} F_{\dc{T}}.
\]

\section{Quasi-arrays}\label{sec:quasiarrays}

This section is devoted to introducing a combinatorial object that explains the existence of unweighted unlabelled
digraph isomorphism between certain connected components of $\Gamma(\hypo)$ or $\Gamma(\hypo_n)$.

\begin{definition}\label{def:quasiarray}
A \defterm{quasi-array} of size $m$ is an array of cells $Q$, with each cell containing an entry from $\aA$, having $m$
rows and columns, that satisfies the following conditions:
\begin{itemize}
  \item[(A1)] For $i = 1,\ldots,m$, the $i$-th row has $m-i+1$ cells.
  \item[(A2)] The entries in the first row are weakly increasing from left to right. That is, $Q_{(1,j)} \leq Q_{(1,j+1)}$ for
        $j = 1,\ldots,m-1$.
  \item[(A3)] In each diagonal from upper right to lower left, the entries are an increasing sequence of consecutive
        elements of $\aA$. That is, $Q_{(i,j)} = Q_{(i-1,j+1)} + 1$ (or, equivalently,
        $Q_{(i,j)} = Q_{(1,j+i-1)} + i - 1$) for $i = 2,\ldots,m$ and $1 \leq j \leq i-1$.
\end{itemize}
\end{definition}

For instance, the following is a quasi-array:
\begin{equation}
  \label{eq:quasi-array-example}
  \tableau{
    2 \& 3 \& 3 \& 5 \& 8 \\
    4 \& 4 \& 6 \& 9 \\
    5 \& 7 \& 10 \\
    8 \& 11 \\
    12 \\
  }
\end{equation}

It is an immediate consequence of conditions (A2) and (A3) that the entries in each row of a quasi-array are weakly
increasing from left to right, and that the entries in each column of a quasi-array are strictly increasing from top to
bottom.

Thus choosing any quasi-ribbon shape within a quasi-array yields a quasi-ribbon tableau. For a given quasi-array $Q$ of
size $m$ and a composition $\sigma$ of $m$, let $\pickqrt{Q}{\sigma}$ be the quasi-ribbon tableau obtained by taking the
quasi-ribbon of shape $\sigma$ within $Q$ whose first entry is the top-left-most entry of $Q$. By the choice of $\sigma$
as a composition of $m$, the last entry of $\pickqrt{Q}{\sigma}$ will be on the $m$-th diagonal.

\begin{example}\label{ex:qarray_qribbon}
If $Q$ is
the quasi-array in \eqref{eq:quasi-array-example}, then
\[
  \pickqrt{Q}{\parens{4,1}} =
  \tableau[topalign]{
    2 \& 3 \& 3 \& 5 \\
      \&   \&   \& 9 \\
  }.
  \qquad
  \pickqrt{Q}{\parens{1,2,2}} =
  \tableau[topalign]{
    2 \\
    4 \& 4 \\
      \& 7 \& 10 \\
  }.
\]
\end{example}

On the other hand, a quasi-ribbon tableau $T$ of length $m$ and shape $\sigma$ determines a unique quasi-array
$\genqa{T}$ of size $m$ such that $\pickqrt{\genqa{T}}{\sigma} = T$. The quasi-array $\genqa{T}$ is obtained by taking
an empty quasi-array of size $m$, placing $T$ into it so that the top-left-most entries of $T$ and the array are
aligned, and filling in the rest of the cells so that each diagonal is an increasing sequence of consecutive symbols
from $\aA$. It is easy to see that $Q = \genqa{\pickqrt{Q}{\sigma}}$ for any quasi-array $Q$ of size $m$ and any
composition $\sigma$ of $m$.

\begin{example}\label{ex:qribbon_qarray}
If $T$ is the quasi-ribbon tableau on the left side of \fullref{Example}{ex:qarray_qribbon}, then we recover the quasi-array \eqref{eq:quasi-array-example}:
\begin{equation}\notag
	\genqa{T} = \tableau{
		2 \& 3 \& 3 \& 5 \& {\color{gray} 8}\\
		{\color{gray} 4} \& {\color{gray} 4} \& {\color{gray} 6} \& 9 \\
		{\color{gray} 5} \& {\color{gray} 7} \& {\color{gray} 10}\\
		{\color{gray} 8} \& {\color{gray} 11}\\
		{\color{gray} 12}\\}
\end{equation}

\end{example}

The set of all quasi-arrays is denoted $\setqa$. The set of all quasi-arrays in which the rightmost entry of the first
row is at most $n$ is denoted $\setqa_n$. Thus the rightmost entry of the $i$-th row of a quasi-array in $\setqa_n$ is
at most $n+i-1$. Thus the quasi-array in \eqref{eq:quasi-array-example} lies in $\setqa_n$ for all $n \geq 8$ but not in
$\setqa_7$.

\begin{definition}\label{def:qa_ops}
For $k \in \nset$, define partial operators $\tc_k$ and $\td_k$ on $\setqa$ as follows. Let $Q \in \setqa$ have size
$m$. Then:
\begin{itemize}
  \item The partial operator $\td_m$ is defined on $Q$, and $\td_m(Q)$ is obtained from $Q$ by adding $1$ to every entry
        of the $m$-th diagonal of $Q$. That is,
        \[
          \td_m(Q)_{(i,j)} = \begin{cases}
                               Q_{(i,j)} + 1 & \text{if $j = m -i +1$,} \\
                               Q_{(i,j)} & \text{otherwise,}
                             \end{cases}
        \]
        for $i = 1,\ldots,m$ and $j = 1,\ldots,m-i+1$.

        For $k = 1,\ldots, m-1$, the partial operator $\td_k$ is defined on $Q$ if and only if
        $Q_{(1,k)} < Q_{(1,k+1)}$; in this case, $\td_k(Q)$ is obtained from $Q$ by adding $1$ to every entry of the
        $k$-th diagonal of $Q$. That is,
        \[
          \td_k(Q)_{(i,j)} = \begin{cases}
                               Q_{(i,j)} + 1 & \text{if $j = k -i +1$,} \\
                               Q_{(i,j)} & \text{otherwise,}
                             \end{cases}
        \]
        for $i = 1,\ldots,m$ and $j = 1,\ldots,m-i+1$.

        For $k > m$, the partial operator $\td_k$ is undefined on $Q$.

  \item For $k = 2,\ldots,m$, the partial operator $\tc_k$ is defined on $Q$ if and only if $Q_{(1,k)} > Q_{(1,k-1)}$; in
        this case, $\tc_k(Q)$ is obtained from $Q$ by subtracting $1$ from every entry of the $k$-th diagonal of $Q$. That is,
        \[
          \tc_k(Q)_{(i,j)} = \begin{cases}
                               Q_{(i,j)} - 1 & \text{if $j = k -i +1$,} \\
                               Q_{(i,j)} & \text{otherwise.}
                             \end{cases}
        \]
        for $i = 1,\ldots,m$ and $j = 1,\ldots,m-i+1$.

        The partial operator $\tc_1$ is defined on $Q$ if and only if $Q_{(1,1)} > 1$; in this case, $\tc_1(Q)$ is
        obtained from $Q$ by subtracting $1$ from the $(1,1)$-th entry of $Q$ (which is the unique entry on the first
        diagonal). That is,
        \[
          \tc_1(Q)_{(i,j)} = \begin{cases}
                               Q_{(i,j)} - 1 & \text{if $i = j = 1$,} \\
                               Q_{(i,j)} & \text{otherwise,}
                             \end{cases}
        \]
        for $i = 1,\ldots,m$ and $j = 1,\ldots,m-i+1$.

        For $k > m$, the partial operator $\tc_k$ is undefined on $Q$.

\end{itemize}
\end{definition}

\begin{example}
	Considering the quasi-array $Q$ in \eqref{eq:quasi-array-example}, of size 5. Adding 1 to all the entries of the second diagonal, one obtains a diagram that is not a quasi-array, as the first row is not weakly increasing, therefore, $\td_2$ is undefined on $Q$. Now, subtracting 1 to all the entries of this diagonal yields a valid quasi-array, hence
	$$\tc_2 (Q) = \tableau{2 \& {\bf 2} \& 3 \& 5 \& 8\\
						   {\bf 3} \& 4 \& 6 \& 9\\
					   	   5 \& 7 \& 10\\
				   	       8 \& 11\\
			   	           12 \\}.$$
	Likewise, subtracting 1 to the entries of the third diagonal does not yield a valid quasi-array, hence $\tc_3$ is undefined on $Q$, while adding 1 to the entries of this diagonal, we obtain
	$$\td_3 (Q) = \tableau{2 \& 3 \& {\bf 4} \& 5 \& 8\\
		4 \& {\bf 5} \& 6 \& 9\\
		{\bf 6} \& 7 \& 10\\
		8 \& 11\\
		12 \\}.$$
\end{example}

\begin{lemma}
  \label{lem:diagonal-operators-closure}
  Let $k \in \nset$ and $Q \in \setqa$.
  \begin{enumerate}
    \item If $\tc_k$ is defined on $Q$, then $\tc_k(Q) \in \setqa$.
    \item If $\td_k$ is defined on $Q$, then $\td_k(Q) \in \setqa$.
  \end{enumerate}
\end{lemma}

\begin{proof}
  Suppose $\tc_k$ is defined on $Q$. Note that $\tc_k(Q)$ has the same shape as $Q$, so $\tc_k(Q)$ satisfies condition
  (A1) in the definition of a quasi-array.

  Consider first the case where $k=1$. Then $Q_{(1,1)} > 1$, and entries of $\tc_1(Q)$ are equal to those of $Q$ except
  that $\tc_k(Q)_{(1,1)} = Q_{(1,1)} - 1 \in \aA$. Since rows of $Q$ are weakly increasing from left to right, it
  follows that rows of $\tc_1(Q)$ are weakly increasing left to right, so $\tc_1(Q)$ satisfies condition (A2). The first
  diagonal of $\tc_1(Q)$ comprises only one entry and the other diagonals equal those of $Q$, so $\tc_1(Q)$ satisfies
  condition (A3).

  Now suppose $k \geq 2$. Then $Q_{(1,k-1)} < Q_{(1,k)}$, and $\tc_k(Q)$ is obtained from $Q$ by subtracting $1$
  from every entry of the $k$-th diagonal of $Q$. Since the $k$-th diagonal of $Q$ is a increasing sequence of
  consecutive elements of $\aA$, so is the $k$-th diagonal of $\tc_k(Q)$. All other diagonals of $Q$ and $\tc_k(Q)$ are
  equal, so $\tc_k$ satisfies condition (A3).

  To prove that $\tc_k(Q)$ satisfies condition (A2), it is necessary to show that the first row of $\tc_k(Q)$ is weakly
  increasing from left to right. The only place this could fail is between entries on the $(k-1)$-th and $k$-th
  diagonals. But $\tc_k(Q)_{(1,k-1)} = Q_{(1,k-1)} \leq Q_{(1,k)} - 1 = \tc_k(Q)_{(1,k)}$, and so $\tc_k(Q)$ satisfies
  condition (A2).

  Hence $\tc_k(Q)$ is a quasi-array.

  Similar reasoning shows that if $\td_k$ is defined on $Q$, then $\td_k(Q)$ is a quasi-array.
\end{proof}

The following result is immediate, and essentially says that $\tc_k$ and $\td_k$ are mutually inverse when they are
defined.

\begin{lemma}
  Let $k \in \nset$ and $Q \in \setqa$.
  \begin{enumerate}
    \item If $\tc_k$ is defined on $Q$, then $\td_k$ is defined on $\tc_k(Q)$ and $\td_k\tc_k(Q) = Q$.
    \item If $\td_k$ is defined on $Q$, then $\tc_k$ is defined on $\td_k(Q)$ and $\tc_k\td_k(Q) = Q$.
  \end{enumerate}
\end{lemma}

The following result establishes the connection between the $\tc_k$ and $\td_k$ and the quasi-Kashiwara operators.

\begin{proposition}
  \label{prop:diagonal-quasi-kashiwara-operators}
  Let $Q \in \setqa$ be of size $m$. Let $\sigma$ be a composition of $m$. Let $k \in \set{1,\ldots,m}$. Suppose the
  $k$-th diagonal of $Q$ intersects $T = \pickqrt{Q}{\sigma}$ at a cell containing $\ell$. Then
  \begin{enumerate}
    \item $\tc_k$ is defined on $Q$ if and only if $\e_{\ell-1}$ is defined on $T$ and alters the symbol $\ell$ on the
          $k$-th diagonal of $T$. In this case, $\e_{\ell-1}(T) = \pickqrt{\tc_k(Q)}{\sigma}$.
    \item $\td_k$ is defined on $Q$ if and only if $\f_\ell$ is defined on $T$ and alters the symbol $\ell$ on the
          $k$-th diagonal of $T$. In this case, $\f_{\ell}(T) = \pickqrt{\td_k(Q)}{\sigma}$.
  \end{enumerate}
\end{proposition}

\begin{proof}
  Only part~(1) is proved here; part~(2) is similar.


  Suppose $\tc_k$ is defined on $Q$. Suppose $\colreading{T}$ contains an $\ell(\ell-1)$-inversion. Then, by the definition of a
  quasi-ribbon tableau, $Q_{(i-1,j)} = T_{(i-1,j)} = \ell-1$. Hence
  $Q_{(1,k)} = Q_{(1,i+j-1)} = Q_{(i-1,j)} - (i-2) = \ell-i+1$ and
  $Q_{(1,k-1)} = Q_{(1,i+j-2)} = Q_{(i,j-1)} -(i-1) = \ell-i+1$, which contradicts $\tc_k$ being defined on $Q$. Thus,
  since $T$ contains a symbol $\ell$, the operator $\e_{\ell-1}$ is defined on $T$.

  Now suppose that the symbol $\ell$ in cell $(i,j)$ is not the left-most such symbol in $T$. Then, by the definition of
  a quasi-ribbon tableau, $Q_{(i,j-1)} = T_{(i,j-1)} = \ell$. Hence,
  $Q_{(1,k)} = Q_{(1,i+j-1)} = Q_{(i,j)} - (i-1) = \ell-i+1$ and
  $Q_{(1,k-1)} = Q_{(1,i+j-2)} = Q_{(i,j-1)} - (i-1) = \ell-i+1$, which again contradicts $\tc_k$ being defined on $Q$.
  Hence $\e_{\ell-1}$ alters the symbol $\ell$ in the $(i,j)$-th cell of $T$.

  For the converse, suppose that $\e_{\ell-1}$ is defined on $T$ and alters the symbol $\ell$ in the $(i,j)$-th cell of
  $T$. Then the $(i,j)$-th cell contains the leftmost appearance of the symbol $\ell$ in $T$. Consider two cases:
  \begin{enumerate}
    \item Suppose that the $(i,j)$-th cell of $T$ does not lie at the leftmost end of a row of $T$. Then
          $Q_{(i,j-1)} = T_{(i,j-1)} < \ell$. Hence $Q_{(1,k)} = Q_{(1,i+j-1)} = Q_{(i,j)} - (i-1) = \ell-i+1$ and
          $Q_{(1,k-1)} = Q_{(1,i+j-2)} = Q_{(i,j-1)} - (i-1) < \ell-i+1$ and so $\tc_k$ is defined on $Q$.
    \item Suppose that the $(i,j)$-th cell of $T$ lies at the leftmost end of a row of $T$. Then the cell above cannot
          contain $\ell-1$, for otherwise $w$ would contain an $\ell(\ell-1)$-inversion. That is,
          $Q_{(i-1,j)} < \ell-1$. Hence $Q_{(1,k)} = Q_{(1,i+j-1)} = Q_{(i,j)} - (i-1) = \ell-i+1$ and
          $Q_{(1,k-1)} = Q_{(1,i+j-2)} = Q_{(i-1,j)} - (i-2) < \ell-i+1$ and so $\tc_k$ is defined on $Q$.
  \end{enumerate}

  Finally, suppose that $\tc_k$ is defined on $Q$ and $\e_{\ell-1}$ is defined on $T$ and alters the symbol $\ell$ on
  the $k$-th diagonal. Then since $\e_i(T)$ is obtained from $T$ by replacing the symbol $\ell$ on the $k$-th diagonal
  by $\ell-1$, and since $\tc_k(Q)$ is obtained from $Q$ by decreasing every entry on the $k$-th diagonal by $1$, it
  follows that $\e_{\ell-1}(T) = \pickqrt{\tc_k(Q)}{\sigma}$.
\end{proof}

The \defterm{quasi-array graph} $\Delta(\setqa)$ is defined to have vertex set $\setqa$ and, for each $Q \in \setqa$ and
$k \in \nset$, an edge from $Q$ to $\td_k(Q)$ labelled by $k$ if and only if $\td_k$ is defined on $Q$. Thus the
operator $\td_k$ corresponds to moving forward along edges and the operator $\tc_k$ corresponds to moving backward along
edges.

\begin{figure}[t]
  \begin{tikzpicture}[]
    \begin{scope}[bigcrystal,labelledcrystaledges,x=20mm,y=20mm]
      \begin{scope}
        \matrix[medtableaumatrix] (1111) at (0,0) {1 \& 1 \& 1 \& 1 \\ 2 \& 2 \& 2 \\ 3 \& 3 \\ 4 \\  };
        \matrix[medtableaumatrix] (1112) at ($ (1111) + (0,-1) $) {1 \& 1 \& 1 \& 2 \\ 2 \& 2 \& 3 \\ 3 \& 4 \\ 5 \\  };
        \matrix[medtableaumatrix] (1122) at ($ (1112) + (0,-1) $) {1 \& 1 \& 2 \& 2 \\ 2 \& 3 \& 3 \\ 4 \& 4 \\ 5 \\  };
        \matrix[medtableaumatrix] (1113) at ($ (1112) + (1,-1) $) {1 \& 1 \& 1 \& 3 \\ 2 \& 2 \& 4 \\ 3 \& 5 \\ 6 \\  };
        \matrix[medtableaumatrix] (1222) at ($ (1122) + (0,-1) $) {1 \& 2 \& 2 \& 2 \\ 3 \& 3 \& 3 \\ 4 \& 4 \\ 5 \\  };
        \matrix[medtableaumatrix] (1123) at ($ (1113) + (0,-1) $) {1 \& 1 \& 2 \& 3 \\ 2 \& 3 \& 4 \\ 4 \& 5 \\ 6 \\  };
        \matrix[medtableaumatrix] (2222) at ($ (1222) + (0,-1) $) {2 \& 2 \& 2 \& 2 \\ 3 \& 3 \& 3 \\ 4 \& 4 \\ 5 \\  };
        \matrix[medtableaumatrix] (1223) at ($ (1222) + (1,-1) $) {1 \& 2 \& 2 \& 3 \\ 3 \& 3 \& 4 \\ 4 \& 5 \\ 6 \\  };
        \matrix[medtableaumatrix] (1133) at ($ (1123) + (1,-1) $) {1 \& 1 \& 3 \& 3 \\ 2 \& 4 \& 4 \\ 5 \& 5 \\ 6 \\  };
        \matrix[medtableaumatrix] (2223) at ($ (2222) + (1,-1) $) {2 \& 2 \& 2 \& 3 \\ 3 \& 3 \& 4 \\ 4 \& 5 \\ 6 \\  };
        \matrix[medtableaumatrix] (1233) at ($ (1223) + (1,-1) $) {1 \& 2 \& 3 \& 3 \\ 3 \& 4 \& 4 \\ 5 \& 5 \\ 6 \\  };
        \matrix[medtableaumatrix] (2233) at ($ (2223) + (0,-1) $) {2 \& 2 \& 3 \& 3 \\ 3 \& 4 \& 4 \\ 5 \& 5 \\ 6 \\  };
        \matrix[medtableaumatrix] (1333) at ($ (1233) + (0,-1) $) {1 \& 3 \& 3 \& 3 \\ 4 \& 4 \& 4 \\ 5 \& 5 \\ 6 \\  };
        \matrix[medtableaumatrix] (2333) at ($ (2233) + (1,-1) $) {2 \& 3 \& 3 \& 3 \\ 4 \& 4 \& 4 \\ 5 \& 5 \\ 6 \\  };
        \matrix[medtableaumatrix] (3333) at ($ (2333) + (0,-1) $) {3 \& 3 \& 3 \& 3 \\ 4 \& 4 \& 4 \\ 5 \& 5 \\ 6 \\  };
        \begin{scope}[pos=.35]
          \draw[f4] (1111) to (1112);
          \draw[f3] (1112) to (1122);
          \draw[f4] (1112) to (1113);
          \draw[f2] (1122) to (1222);
          \draw[f4] (1122) to (1123);
          \draw[f3] (1113) to (1123);
          \draw[f1] (1222) to (2222);
          \draw[f4] (1222) to (1223);
          \draw[f2] (1123) to (1223);
          \draw[f3] (1123) to (1133);
          \draw[f4] (2222) to (2223);
          \draw[f1] (1223) to (2223);
          \draw[f3] (1223) to (1233);
          \draw[f2] (1133) to (1233);
          \draw[f3] (2223) to (2233);
          \draw[f1] (1233) to (2233);
          \draw[f2] (1233) to (1333);
          \draw[f2] (2233) to (2333);
          \draw[f1] (1333) to (2333);
          \draw[f1] (2333) to (3333);
        \end{scope}
      \end{scope}
    \end{scope}
  \end{tikzpicture}
  \caption{The connected component $\Delta(\setqa_3,4)$.}
  \label{fig:delta-t3-4}
\end{figure}

The induced subgraph of $\Delta(\setqa)$ containing quasi-arrays of size $m$ is denoted $\Delta(\setqa,m)$; it will be
seen shortly that each $\Delta(\setqa,m)$ is a connected component of $\Delta(\setqa)$. The graph $\Delta(\setqa_n)$ is
the induced subgraph of $\Delta(\setqa)$ with vertex set $\setqa_n$, and the connected component of $\Delta(\setqa_n)$
containing all quasi-arrays in $\setqa_n$ of size $m$ is denoted $\Delta(\setqa_n,m)$ (see
\fullref{Figure}{fig:delta-t3-4}).

\section{Connection to quasi-crystal graphs}\label{sec:connection_qc}
In this section we will introduce explicit isomorphisms between certain quasi-array graphs and quasi-crystal graphs.
This will allow us to determine when certain connected components of quasi-crystal graphs are isomorphic, as unlabelled graphs.
We will cover the infinite and finite rank cases separately, as the hypothesis are slightly different.


\subsection{Infinite-rank case}

The following result is immediate from \fullref{Proposition}{prop:diagonal-quasi-kashiwara-operators}:

\begin{theorem}
  \label{thm:isom-quasi-array-graph-quasi-crystal}
  For any composition $\sigma$ of $m$, the maps
  \begin{align*}
    \Phi_\sigma &: \Delta(\setqa,m) \to \Gamma(\hypo,\sigma),& Q &\mapsto \pickqrt{Q}{\sigma} \\
  \intertext{and}
    \genqalit|_{\Gamma(\hypo,\sigma)} &: \Gamma(\hypo,\sigma) \to \Delta(\setqa,m),& T &\mapsto \genqa{T}
  \end{align*}
  are mutually inverse unlabelled directed graph isomorphisms.
\end{theorem}

Since $\Phi_\sigma$ is an unlabelled directed graph isomorphism, and $\Gamma(\hypo,\sigma)$ is a connected component of
$\Gamma(\hypo)$, it follows that $\Delta(\setqa,m)$ is a connected component of $\Delta(\setqa)$.

Note that the map $\Phi_\sigma$ in \fullref{Theorem}{thm:isom-quasi-array-graph-quasi-crystal} does not preserve labels
except in trivial cases: an edge $Q \cedge[k]{i} Q'$ maps to an edge
$\pickqrt{Q}{\sigma} \cedge[\ell]{i} \pickqrt{Q'}{\sigma}$ where $\ell$ is the entry of $\pickqrt{Q}{\sigma}$ that lies
on the $k$-th diagonal (see \fullref{Figure}{fig:isom-quasi-array-graph-quasi-crystal}).

\begin{figure}[t]
  \begin{tikzpicture}
    \begin{scope}[bigcrystal,labelledcolouredcrystaledges,x=20mm,y=20mm]
      \begin{scope}
        \matrix[medtableaumatrix] (1212) at (0,0) {1 \& 1 \\ \& 2 \\ \& 3 \\ };
        \matrix[medtableaumatrix] (1213) at ($ (1212) + (0,-1) $) {1 \& 1 \\ \& 2 \\ \& 4 \\ };
        \matrix[medtableaumatrix] (1313) at ($ (1213) + (0,-1) $) {1 \& 1 \\ \& 3 \\ \& 4 \\ };
        \matrix[medtableaumatrix] (1214) at ($ (1213) + (1,-1) $) {1 \& 1 \\ \& 2 \\ \& 5 \\ };
        \matrix[medtableaumatrix] (1323) at ($ (1313) + (0,-1) $) {1 \& 2 \\ \& 3 \\ \& 4 \\ };
        \matrix[medtableaumatrix] (1314) at ($ (1214) + (0,-1) $) {1 \& 1 \\ \& 3 \\ \& 5 \\ };
        \matrix[medtableaumatrix] (2323) at ($ (1323) + (0,-1) $) {2 \& 2 \\ \& 3 \\ \& 4 \\ };
        \matrix[medtableaumatrix] (1324) at ($ (1323) + (1,-1) $) {1 \& 2 \\ \& 3 \\ \& 5 \\ };
        \matrix[medtableaumatrix] (1414) at ($ (1314) + (1,-1) $) {1 \& 1 \\ \& 4 \\ \& 5 \\ };
        \matrix[medtableaumatrix] (2324) at ($ (2323) + (1,-1) $) {2 \& 2 \\ \& 3 \\ \& 5 \\ };
        \matrix[medtableaumatrix] (1424) at ($ (1324) + (1,-1) $) {1 \& 2 \\ \& 4 \\ \& 5 \\ };
        \matrix[medtableaumatrix] (2424) at ($ (2324) + (0,-1) $) {2 \& 2 \\ \& 4 \\ \& 5 \\ };
        \matrix[medtableaumatrix] (1434) at ($ (1424) + (0,-1) $) {1 \& 3 \\ \& 4 \\ \& 5 \\ };
        \matrix[medtableaumatrix] (2434) at ($ (2424) + (1,-1) $) {2 \& 3 \\ \& 4 \\ \& 5 \\ };
        \matrix[medtableaumatrix] (3434) at ($ (2434) + (0,-1) $) {3 \& 3 \\ \& 4 \\ \& 5 \\ };
        \draw[f3] (1212) to (1213);
        \draw[f2] (1213) to (1313);
        \draw[f4] (1213) to (1214);
        \draw[f1] (1313) to (1323);
        \draw[f4] (1313) to (1314);
        \draw[f2] (1214) to (1314);
        \draw[f1] (1323) to (2323);
        \draw[f4] (1323) to (1324);
        \draw[f1] (1314) to (1324);
        \draw[f3] (1314) to (1414);
        \draw[f4] (2323) to (2324);
        \draw[f1] (1324) to (2324);
        \draw[f3] (1324) to (1424);
        \draw[f1] (1414) to (1424);
        \draw[f3] (2324) to (2424);
        \draw[f1] (1424) to (2424);
        \draw[f2] (1424) to (1434);
        \draw[f2] (2424) to (2434);
        \draw[f1] (1434) to (2434);
        \draw[f2] (2434) to (3434);
      \end{scope}
      \begin{scope}
        \matrix[medtableaumatrix] (1121) at (4,0) {1 \& 1 \& 1 \\ \& \& 2 \\ };
        \matrix[medtableaumatrix] (1131) at ($ (1121) + (0,-1) $) {1 \& 1 \& 1 \\ \& \& 3 \\ };
        \matrix[medtableaumatrix] (1132) at ($ (1131) + (0,-1) $) {1 \& 1 \& 2 \\ \& \& 3 \\ };
        \matrix[medtableaumatrix] (1141) at ($ (1131) + (1,-1) $) {1 \& 1 \& 1 \\ \& \& 4 \\ };
        \matrix[medtableaumatrix] (1232) at ($ (1132) + (0,-1) $) {1 \& 2 \& 2 \\ \& \& 3 \\ };
        \matrix[medtableaumatrix] (1142) at ($ (1141) + (0,-1) $) {1 \& 1 \& 2 \\ \& \& 4 \\ };
        \matrix[medtableaumatrix] (2232) at ($ (1232) + (0,-1) $) {2 \& 2 \& 2 \\ \& \& 3 \\ };
        \matrix[medtableaumatrix] (1242) at ($ (1142) + (0,-1) $) {1 \& 2 \& 2 \\ \& \& 4 \\ };
        \matrix[medtableaumatrix] (1143) at ($ (1142) + (1,-1) $) {1 \& 1 \& 3 \\ \& \& 4 \\ };
        \matrix[medtableaumatrix] (2242) at ($ (2232) + (1,-1) $) {2 \& 2 \& 2 \\ \& \& 4 \\ };
        \matrix[medtableaumatrix] (1243) at ($ (1242) + (1,-1) $) {1 \& 2 \& 3 \\ \& \& 4 \\ };
        \matrix[medtableaumatrix] (2243) at ($ (2242) + (0,-1) $) {2 \& 2 \& 3 \\ \& \& 4 \\ };
        \matrix[medtableaumatrix] (1343) at ($ (1243) + (0,-1) $) {1 \& 3 \& 3 \\ \& \& 4 \\ };
        \matrix[medtableaumatrix] (2343) at ($ (2243) + (1,-1) $) {2 \& 3 \& 3 \\ \& \& 4 \\ };
        \matrix[medtableaumatrix] (3343) at ($ (2343) + (0,-1) $) {3 \& 3 \& 3 \\ \& \& 4 \\ };
        \draw[f2] (1121) to (1131);
        \draw[f1] (1131) to (1132);
        \draw[f3] (1131) to (1141);
        \draw[f1] (1132) to (1232);
        \draw[f3] (1132) to (1142);
        \draw[f1] (1141) to (1142);
        \draw[f1] (1232) to (2232);
        \draw[f3] (1232) to (1242);
        \draw[f1] (1142) to (1242);
        \draw[f2] (1142) to (1143);
        \draw[f3] (2232) to (2242);
        \draw[f1] (1242) to (2242);
        \draw[f2] (1242) to (1243);
        \draw[f1] (1143) to (1243);
        \draw[f2] (2242) to (2243);
        \draw[f1] (1243) to (2243);
        \draw[f2] (1243) to (1343);
        \draw[f2] (2243) to (2343);
        \draw[f1] (1343) to (2343);
        \draw[f2] (2343) to (3343);
      \end{scope}
    \end{scope}
  \end{tikzpicture}
  \caption{The connected components $\Gamma(\hypo_5,1321)$ and $\Gamma(\hypo_4,1121)$ (which respectively comprise
    quasi-ribbon tableaux in $\hypo_5$ of shape $(2,1,1)$ and quasi-ribbon tableaux in $\hypo_4$ of shapes $(3,1)$).}
  \label{fig:isom-quasi-array-graph-quasi-crystal}
\end{figure}

\begin{corollary}
  \label{corol:isomorphism-unlabelled}
  Let $\sigma$ and $\tau$ be compositions of the same natural number. Then the map
  \begin{equation}
    \label{eq:isomorphism-unlabelled}
    \Psi_{\sigma,\tau} : \Gamma(\hypo,\sigma) \to \Gamma(\hypo,\tau), \qquad T \mapsto \pickqrt{\genqa{T}}{\tau}
  \end{equation}
  is an unlabelled directed graph isomorphism.
\end{corollary}

\begin{proof}
  Suppose $\sigma$ and $\tau$ are both compositions of $m \in \nset$. Then by
  \fullref{Theorem}{thm:isom-quasi-array-graph-quasi-crystal}, the maps
  \begin{align*}
    \genqalit|_{\Gamma(\hypo,\sigma)} & : \Gamma(\hypo,\sigma) \to \Delta(\setqa,m) \\
    \Phi_\tau                         & : \Delta(\setqa,m) \to \Gamma(\hypo,\tau)
  \end{align*}
  are unlabelled directed graph isomorphisms. Their composition is $\Psi_{\sigma,\tau}$, which is thus also an
  unlabelled directed graph isomorphism.
\end{proof}

%


The following results are converses of \fullref{Theorem}{thm:isom-quasi-array-graph-quasi-crystal} and
\fullref{Corollary}{corol:isomorphism-unlabelled}, respectively:


\begin{proposition}
  \label{prop:isom-quasi-array-graph-quasi-crystal-converse}
  Let $m_1,m_2 \in \nset$. If $\Delta(\setqa,m_1)$ and $\Delta(\setqa,m_2)$ are isomorphic as (unlabelled) directed
  graphs, then $m_1 = m_2$.
\end{proposition}

\begin{proof}
  Since the graph $\Delta(\setqa,m_1)$ describes the action of the partial operators $\td_i$ for
  $i \in \set{1,\ldots,m_1}$ on quasi-arrays of size $m_1$, it follows that the outdegree of a vertex in
  $\Delta(\setqa,m_1)$ must be less than or equal to $m_1$. There are vertices with this outdegree: all that is
  necessary for all the $\td_i$ to be defined is for there to be no equal entries in the first row. Hence the
  (unlabelled) directed graph structure of $\Delta(\setqa,m_1)$ determines $m_1$. The result is now immediate.
\end{proof}

\begin{corollary}
  \label{corol:isom-implies-equal-composition-size}
  Let $\sigma_1$ and $\sigma_2$ be compositions of $m_1$ and $m_2$. If $\Gamma(\hypo,\sigma_1)$ and
  $\Gamma(\hypo,\sigma_2)$ are isomorphic as (unlabelled) directed graphs, then $m_1 = m_2$.
\end{corollary}

\begin{proof}
  Since $\Gamma(\hypo,\sigma_1)$ is isomorphic to $\Delta(\setqa,m_1)$, and $\Gamma(\hypo,\sigma_2)$ is isomorphic to
  $\Delta(\setqa,m_2)$ by \fullref{Theorem}{thm:isom-quasi-array-graph-quasi-crystal}, the result follows immediately
  from \fullref{Proposition}{prop:isom-quasi-array-graph-quasi-crystal-converse}.
\end{proof}

\subsection{Finite-rank case}

The maximum entry in $\pickqrt{Q}{\sigma}$ is the maximum entry in the $\cwt\sigma$-th row of $Q$. Hence if $Q$ lies in
$\setqa_n$, then $\pickqrt{Q}{\sigma}$ lies in $\hypo_{n+\cwt\sigma-1}$. Hence there are similar isomorphisms of
unlabelled directed graphs between $\Delta(\setqa_n,m)$ and $\Gamma(\hypo_{n+\cwt\sigma-1},\sigma)$, and so there hold the
following finite-rank analogues of \fullref{Theorem}{thm:isom-quasi-array-graph-quasi-crystal} and
\fullref{Corollary}{corol:isomorphism-unlabelled}:

\begin{theorem}
  \label{thm:isom-quasi-array-graph-quasi-crystal-finite-rank}
  For any composition $\sigma$ of $m$, the maps
  \begin{align*}
    \Phi_\sigma &: \Delta(\setqa_n,m) \to \Gamma(\hypo_{n+\cwt\sigma-1},\sigma),& Q &\mapsto \pickqrt{Q}{\sigma} \\
  \intertext{and}
    \genqalit|_{\Gamma(\hypo_{n+\cwt\sigma-1},\sigma)} &: \Gamma(\hypo_{n+\cwt\sigma-1},\sigma) \to \Delta(\setqa_n,m),& T &\mapsto \genqa{T}
  \end{align*}
  are mutually inverse unlabelled directed graph isomorphisms.
\end{theorem}

\begin{corollary}
  \label{corol:isomorphism-unlabelled-finite-rank}
  Let $\sigma$ and $\tau$ be compositions of the same natural number. Then for any
  $n > \abs[\big]{\cwt\sigma - \cwt\tau}$, the map
  \begin{equation}
    \label{eq:isomorphism-unlabelled-finite-rank}
    \Psi_{\sigma,\tau} : \Gamma(\hypo_n,\sigma) \to \Gamma(\hypo_{n- \abs*{\cwt\sigma - \cwt\tau}},\tau), \qquad T \mapsto \pickqrt{\genqa{T}}{\tau}
  \end{equation}
  is an unlabelled directed graph isomorphism.
\end{corollary}

As a direct consequence, if $\sigma$ and $\tau$ have the same number of parts, we have the following finite-rank analogue of \fullref{Corollary}{corol:isomorphism-unlabelled}:

\begin{corollary}
  \label{corol:isom-compositions-finite-rank}
  Let $\sigma$ and $\tau$ be compositions of $m \in \nset$ with the same number of parts (that is, with
  $\cwt\sigma = \cwt\tau$). Then for any $n \in \nset$, the connected components $\Gamma(\hypo_n,\sigma)$ and
  $\Gamma(\hypo_n,\tau)$ are isomorphic as unlabelled directed graphs under the map
  \eqref{eq:isomorphism-unlabelled-finite-rank}.
\end{corollary}

We recall the classical \emph{Schützenberger involution}, now restricted to quasi-crystal graphs \cite[\S~8.3]{cm_hypoplactic}, a map $\natural : \aA_n^* \to \aA_n^*$ that sends each letter $a \in \aA_n$ to $n-a+1$ and is extended to $\aA_n^*$ by $(a_1 \ldots a_k)^\natural = a_k^\natural \ldots a_1^\natural$.
Loosely speaking, the following result says that, viewed as unlabelled undirected graphs, connected components of
$\Gamma(\hypo_n)$ are `vertically symmetrical'.

\begin{corollary}
  \label{corol:vertically-symmetrical}
  For each connected component of $\Gamma(\hypo_n)$, there is an unlabelled undirected graph automorphism of that
  component that maps the unique highest weight vertex to the unique lowest weight vertex.
\end{corollary}

\begin{proof}
  Let $u,v \in \aA_n^*$. Since $\f_i(u) = v$ if and only if $\f_{n-i+1}(v^\natural) = u^\natural$, the Schützenberger
  involution ${\cdot} \mapsto {\cdot}^\natural$ induces an isomorphism of undirected graphs from $\Gamma(\hypo_n,u)$ to
  $\Gamma(\hypo_n,u^\natural)$ that maps the highest-weight vertex of $\Gamma(\hypo_n,u)$ to the lowest weight vertex of
  $\Gamma(\hypo_n,u^\natural)$ \cite[Propostion~20]{cm_hypoplactic}.

  Furthermore, the shape of $\phypo{u^\natural}$ is obtained by rotating the shape of $\phypo{u}$ through a half-turn.
  In particular, $\phypo{u}$ and $\phypo{u^\natural}$ have the same number of rows. Hence $\Gamma(\hypo_n,u)$ and
  $\Gamma(\hypo_n,u^\natural)$ are isomorphic as unlabelled directed graphs by
  \fullref{Corollary}{corol:isom-compositions-finite-rank}.

  Combining these two isomorphisms gives the required undirected graph automorphism of $\Gamma(\hypo_n,u)$.
\end{proof}

In the infinite-rank case, \fullref{Corollary}{corol:isom-implies-equal-composition-size} is a converse of
\fullref{Corollary}{corol:isomorphism-unlabelled}. The converse of the finite-rank analogue of the latter result,
\fullref{Corollary}{corol:isom-compositions-finite-rank}, does not hold, in the sense that the length of a quasi-ribbon
shape (that is, the size of a compositon $\sigma$) is not determined by the isomorphism type of the component
$\Gamma(\hypo_n,\sigma)$. This is immediate from the fact that every word $u \in \aA_n^*$ such that the quasi-ribbon
tableau $\phypo{u}$ has $n$ rows must contain an $(i+1)i$-inversion of all $i \in \set{1,\ldots,n-1}$ and so is an
isolated vertex in $\Gamma(\hypo_n)$.

\subsection{Isomorphisms and calculation}
\label{subsec:isomorphism-calculation}

Note that \fullref{Corollaries}{corol:isomorphism-unlabelled} and \ref{corol:isomorphism-unlabelled-finite-rank} say
respectively that connected components of $\Gamma(\hypo)$ and $\Gamma(\hypo_n)$ corresponding to quasi-ribbon shapes
with the same length are isomorphic as unlabelled directed graphs under the maps \eqref{eq:isomorphism-unlabelled} and
\eqref{eq:isomorphism-unlabelled-finite-rank}.

The existence of these isomorphisms is a consequence of a result of Maas-Gariépy
\cite[Theorem~2]{maasgariepy_quasicrystal}, but the notion of quasi-arrays gives an explicit isomorphism. Indeed,
quasi-arrays make it very easy to compute the map, without actually having to compute the whole quasi-array. For
instance, let $m = 5$, $\sigma = (1,4)$, $\tau = (2,1,1,1)$, and consider computing the image of a particular
quasi-ribbon tableau via computing the whole quasi-array:
\[
  \Phi_\tau\genqalit|_{\Gamma(\hypo_{9},\sigma)}\parens*{
    \onmathaxis{
      \begin{tikzpicture}[baseline=(maintableau-1-1.south)]
        \matrix[pickmatrix] (maintableau) {
          |[drawnode]| 2 \\
          |[drawnode]| 4 \& |[drawnode]| 4 \& |[drawnode]| 6 \& |[drawnode]| 9 \\
        };
      \end{tikzpicture}
    }
  }
  =
  \Phi_\tau\parens*{
    \onmathaxis{
      \begin{tikzpicture}[baseline=(maintableau-3-1.center)]
        \matrix[tableaumatrix] (maintableau) {
          2 \& 3 \& 3 \& 5 \& 8 \\
          4 \& 4 \& 6 \& 9 \\
          5 \& 7 \& 10 \\
          8 \& 11 \\
          12 \\
        };
      \end{tikzpicture}
    }
  }
  =
  \onmathaxis{
    \begin{tikzpicture}[baseline=(maintableau-2-1.south)]
      \matrix[pickmatrix] (maintableau) {
        |[drawnode]| 2 \& |[drawnode]| 3 \\
        \null \& |[drawnode]| 4 \\
        \null \& |[drawnode]| 7 \\
        \null \& |[drawnode]| 11 \\
      };
    \end{tikzpicture}
  }
\]
But all one has to do is consider the difference in distance between entries along the same diagonals and
subtract (on moving upwards and rightwards) or add (on moving downwards and leftwards):
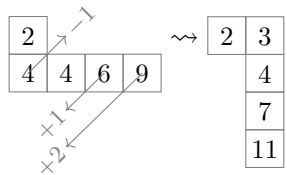
\[
  \begin{tikzpicture}[baseline=(maintableau-1-1.base)]
    \matrix[pickmatrix] (maintableau) {
      |[drawnode]| 2 \& \null \\
      |[drawnode]| 4 \& |[drawnode]| 4 \& |[drawnode]| 6 \& |[drawnode]| 9 \\
      \null \& \null \\
      \null \& \null \\
    };
    \begin{scope}[
      ->,
      draw=gray,
      line cap=round,
      every node/.style={
        text=gray,
        inner sep=2pt,
        outer sep=0,
        node font=\footnotesize,
      }
      ]
      \draw (maintableau-2-1.center) -- node[pos=1,anchor=west,rotate=45] {$-1$} (maintableau-1-2.center);
      \draw (maintableau-2-3.center) -- node[pos=1,anchor=east,rotate=45] {$+1$} (maintableau-3-2.center);
      \draw (maintableau-2-4.center) -- node[pos=1,anchor=east,rotate=45] {$+2$} (maintableau-4-2.center);
    \end{scope}
  \end{tikzpicture}
  \leadsto
  \begin{tikzpicture}[baseline=(maintableau-1-1.base)]
    \matrix[pickmatrix] (maintableau) {
      |[drawnode]| 2 \& |[drawnode]| 3 \\
      \null \& |[drawnode]| 4 \\
      \null \& |[drawnode]| 7 \\
      \null \& |[drawnode]| 11 \\
    };
  \end{tikzpicture}
\]

\subsection{Loop version of quasi-crystals}

In the modified notion of quasi-crystal used in \cite{cgm_quasicrystals}, at each vertex containing an subsequence
$(i+1)i$ (where neither $\e_i$ nor $\f_i$ is defined) there is a loop labelled by $i$. With this kind of quasi-crystal,
the analogues of \fullref{Corollaries}{corol:isomorphism-unlabelled} and \ref{corol:isomorphism-unlabelled-finite-rank}
do not hold; see \fullref{Figure}{fig:inversion-loop-version}.

\begin{figure}[t]
  \begin{tikzpicture}
    \begin{scope}[bigcrystal,labelledcolouredcrystaledges]
      \begin{scope}
        \matrix[medtableaumatrix] (1121) at (0,0) {1 \& 1 \& 1 \\ \& \& 2 \\ };
        \matrix[medtableaumatrix] (1131) at ($ (1121) + (0,-1) $) {1 \& 1 \& 1 \\ \& \& 3 \\ };
        \matrix[medtableaumatrix] (1132) at ($ (1131) + (0,-1) $) {1 \& 1 \& 2 \\ \& \& 3 \\ };
        \matrix[medtableaumatrix] (1141) at ($ (1131) + (1,-1) $) {1 \& 1 \& 1 \\ \& \& 4 \\ };
        \begin{pgfinterruptboundingbox}
          \node[minimum size=10mm] (1232) at ($ (1132) + (0,-1) $) {};
          \node[minimum size=10mm] (1142) at ($ (1141) + (0,-1) $) {};
        \end{pgfinterruptboundingbox}
        \draw[f2] (1121) to (1131);
        \draw[f1] (1131) to (1132);
        \draw[f3] (1131) to (1141);
        \draw[f1] (1132) to (1232);
        \draw[f3] (1132) to (1142);
        \draw[f1] (1141) to (1142);
        \draw[f1,loop left] (1121) to (1121);
        \draw[f2,loop left] (1132) to (1132);
      \end{scope}
      \begin{scope}
        \matrix[medtableaumatrix] (1212) at (4,0) {1 \& 1 \\ \& 2 \& 2 \\ };
        \matrix[medtableaumatrix] (1213) at ($ (1212) + (0,-1) $) {1 \& 1 \\ \& 2 \& 3 \\ };
        \matrix[medtableaumatrix] (1313) at ($ (1213) + (0,-1) $) {1 \& 1 \\ \& 3 \& 3 \\ };
        \matrix[medtableaumatrix] (1214) at ($ (1213) + (1,-1) $) {1 \& 1 \\ \& 2 \& 4 \\ };
        \begin{pgfinterruptboundingbox}
          \node[minimum size=10mm] (1323) at ($ (1313) + (0,-1) $) {};
          \node[minimum size=10mm] (1314) at ($ (1214) + (0,-1) $) {};
        \end{pgfinterruptboundingbox}
        \draw[f2] (1212) to (1213);
        \draw[f2] (1213) to (1313);
        \draw[f3] (1213) to (1214);
        \draw[f1] (1313) to (1323);
        \draw[f3] (1313) to (1314);
        \draw[f2] (1214) to (1314);
        \draw[f1,loop left] (1212) to (1212);
        \draw[f1,loop left] (1213) to (1213);
        \draw[f1,loop right] (1214) to (1214);
      \end{scope}
    \end{scope}
  \end{tikzpicture}
  \caption{The connected components $\Gamma(\hypo_4,1121)$ and $\Gamma(\hypo_4,1212)$ using the notion of quasi-crystals
    used in \cite{cgm_quasicrystals}. The two shapes of quasi-ribbon tableaux have the same length, but the loops
    indicating the presence of $(i+1)i$-inversions mean that the graphs are not isomorphic as (unlabelled) directed
    graphs.}
  \label{fig:inversion-loop-version}
\end{figure}

\section{Geometry of the quasi-crystal graph}\label{sec:geometry}

The isomorphisms between connected components of $\Gamma(\hypo)$ or $\Gamma(\hypo_n)$ suggest a natural geometric
interpretation of connected components as lattices in $m$-dimensional space, where $m$ is the length of the words (or
quasi-ribbon tableaux) contained in that component.

When the quasi-ribbon tableau has the row shape $(m)$, one can view the entries of the quasi-ribbon tableau
$\tableau{x_1 \& x_2 \& |[dottedentry]| \cdots \& x_m \\}$ as coordinates $\parens{x_1,x_2,\ldots,x_m}$ in
$m$-dimensional space. In this case, the action of the operator $\f_i$, when defined, increases some coordinate $x_h$
from $i$ to $i+1$. Thus each edge between a vertex in the hyperplane $x_h = i$ and a vertex in the hyperplane
$x_h = i+1$ is labelled by $i$.

For $\Gamma(\hypo,(m))$, a lattice element $(x_1,x_2,\ldots,x_m)$ corresponds to a quasi-ribbon tableau if and only if
it satisfies $1 \leq x_1 \leq x_2 \leq \ldots \leq x_m$, which is equivalent to the coordinates being the entries of a
quasi-ribbon tableau with a single row. The (infinite) convex hull of lattice elements corresponding to quasi-ribbon
tableaux is bounded by the $n-1$ hyperplanes (each of dimension $n-1$) defined by $x_1 = 1$ and by $x_h = x_{h+1}$ for
some $h$.

For $\Gamma(\hypo_n,(m))$, a lattice element $(x_1,x_2,\ldots,x_m)$ corresponds to a quasi-ribbon tableau if and only if
it satisfies $1 \leq x_1 \leq x_2 \leq \ldots \leq x_m \leq n$. (Moving from the infinite-rank to the finite-rank case
adds the additional condition $x_m \leq n$.) The (finite) convex hull of lattice elements corresponding to quasi-ribbon
tableaux is bounded by the $n$ hyperplanes (each of dimension $n-1$) defined by $x_1 = 1$, by $x_h = x_{h+1}$ for some
$h$, and (unlike for $\Gamma(\hypo,(m))$) by $x_m = n$. \fullref{Figure}{fig:gamma-hypo-4-3} shows $\Gamma(\hypo_n,(m))$
as a lattice comprising points in three-dimensional space. In this example, the convex hull is tetrahedral, bounded by
four (two-dimensional) planes, on each of which there are ten lattice points corresponding to quasi-ribbon tableaux:
\begin{itemize}
  \item the plane $x_1 = 1$, which passes through the ten vertices of the form $\tableau{1 \& x_2 \& x_3 \\}$,
        $\tableau{1 \& x_2 \& x_3 \\}$, $\tableau{1 \& x_2 \& x_3 \\}$, $\tableau{1 \& x_2 \& x_3 \\}$ at the `front' of
        the diagram;
  \item the plane $x_1 = x_2$, which passes through the ten vertices of the form $\tableau{1 \& 1 \& x_3 \\}$,
        $\tableau{2 \& 2 \& x_3 \\}$, $\tableau{3 \& 3 \& x_3 \\}$, $\tableau{4 \& 4 \& x_3 \\}$ at the angled `back' of
        the diagram;
  \item the plane $x_2 = x_3$, which passes through the ten vertices of the form $\tableau{x_1 \& 1 \& 1 \\}$,
        $\tableau{x_1 \& 2 \& 2 \\}$, $\tableau{x_1 \& 3 \& 3 \\}$, $\tableau{x_1 \& 4 \& 4 \\}$ at the angled `top' of
        the diagram;
  \item the plane $x_3 = 4$, which passes through the ten vertices of the form $\tableau{x_1 \& x_2 \& 4 \\}$,
        $\tableau{x_1 \& x_2 \& 4 \\}$, $\tableau{x_1 \& x_2 \& 4 \\}$, $\tableau{x_1 \& x_2 \& 4 \\}$ at the `bottom'
        of the diagram.
\end{itemize}

\begin{figure}[t]
  \begin{tikzpicture}

    \begin{scope}[
      bigcrystal,
      labelledcolouredcrystaledges,
      x={(16mm,20mm)},
      y={(18mm,0)},
      z={(0mm,-14mm)},
      ]
      \begin{scope}

        \node (111) at (1,1,1) {\smalltableau{1 \& 1 \& 1 \\}};

        \node (112) at (1,1,2) {\smalltableau{1 \& 1 \& 2 \\}};
        \node (122) at (1,2,2) {\smalltableau{1 \& 2 \& 2 \\}};
        \node (222) at (2,2,2) {\smalltableau{2 \& 2 \& 2 \\}};

        \node (113) at (1,1,3) {\smalltableau{1 \& 1 \& 3 \\}};
        \node (123) at (1,2,3) {\smalltableau{1 \& 2 \& 3 \\}};
        \node (223) at (2,2,3) {\smalltableau{2 \& 2 \& 3 \\}};
        \node (133) at (1,3,3) {\smalltableau{1 \& 3 \& 3 \\}};
        \node (233) at (2,3,3) {\smalltableau{2 \& 3 \& 3 \\}};
        \node (333) at (3,3,3) {\smalltableau{3 \& 3 \& 3 \\}};

        \node (114) at (1,1,4) {\smalltableau{1 \& 1 \& 4 \\}};
        \node (124) at (1,2,4) {\smalltableau{1 \& 2 \& 4 \\}};
        \node (224) at (2,2,4) {\smalltableau{2 \& 2 \& 4 \\}};
        \node (134) at (1,3,4) {\smalltableau{1 \& 3 \& 4 \\}};
        \node (234) at (2,3,4) {\smalltableau{2 \& 3 \& 4 \\}};
        \node (334) at (3,3,4) {\smalltableau{3 \& 3 \& 4 \\}};
        \node (144) at (1,4,4) {\smalltableau{1 \& 4 \& 4 \\}};
        \node (244) at (2,4,4) {\smalltableau{2 \& 4 \& 4 \\}};
        \node (344) at (3,4,4) {\smalltableau{3 \& 4 \& 4 \\}};
        \node (444) at (4,4,4) {\smalltableau{4 \& 4 \& 4 \\}};

        \draw[f1] (111) to (112);

        \draw[f2] (112) to (113);
        \draw[f2] (122) to (123);
        \draw[f2] (222) to (223);

        \draw[f3] (113) to (114);
        \draw[f3] (123) to (124);
        \draw[f3] (223) to (224);
        \draw[f3] (133) to (134);
        \draw[f3] (233) to (234);
        \draw[f3] (333) to (334);

        \draw[f1] (112) to (122);
        \draw[f1] (113) to (123);
        \draw[f1] (114) to (124);

        \draw[f2] (123) to (133);
        \draw[f2] (223) to (233);
        \draw[f2] (124) to (134);
        \draw[f2] (224) to (234);

        \draw[f3] (134) to (144);
        \draw[f3] (234) to (244);
        \draw[f3] (334) to (344);

        \draw[f1] (122) to (222);
        \draw[f1] (123) to (223);
        \draw[f1] (133) to (233);
        \draw[f1] (124) to (224);
        \draw[f1] (134) to (234);
        \draw[f1] (144) to (244);

        \draw[f2] (233) to (333);
        \draw[f2] (234) to (334);
        \draw[f2] (244) to (344);

        \draw[f3] (344) to (444);

      \end{scope}
    \end{scope}
  \end{tikzpicture}
  \caption{The connected component $\Gamma(\hypo_4,(3))$, drawn as a three-dimensional lattice, placing each entry of
    the quasi-ribbon tableau $\tableau{x_1 \& x_2 \& x_3 \\}$ at the point $(x_1,x_2,x_3)$.}
  \label{fig:gamma-hypo-4-3}
\end{figure}

Similarly, one can place the vertices of $\Gamma(\hypo,\sigma)$ (where $\sigma$ is a composition of $m$) as lattice
elements by treating each entry of a coordinate. Furthermore, $\Gamma(\hypo,\sigma)$ is isomorphic as an unlabelled
directed graphs to $\Gamma(\hypo,(m))$, and the explicit description of the isomorphism in
\fullref{Subsection}{subsec:isomorphism-calculation} indicates that the entries in each vertex of can be obtained by
adding fixed amounts to the the entries of the corresponding vertex of $\Gamma(\hypo,(m))$. This corresponds to
shifting each lattice point by a fixed amount. Thus $\Gamma(\hypo,\sigma)$ is obtained by shifting $\Gamma(\hypo,(m))$
by a fixed amount in $m$-dimensional space, and relabelling edges when necessary so that an edge from a vertex in the
hyperplane $x_h = i$ to one in $x_h = i+1$ is labelled by $i$.

Applying an analogous shift and relabelling to $\Gamma(\hypo_n,(m))$ yields $\Gamma(\hypo_n,\sigma)$; see
\fullref{Figure}{fig:gamma-hypo-4-21}. This
argument proves the following result:

\begin{proposition}
	Let $m \geq 1$ and $n \geq 1$. The connected component
	$\Gamma(\hypo_n,(m))$ of the quasi-crystal graph is isomorphic to
	the labelled directed graph whose vertices are integer points in the
	convex polytope
	\[
	\mathcal{P}_{n,m} := \operatorname{conv}\left\{ (x_1, x_2, \ldots,
	x_m) \in \mathbb{Z}^m : 1 \leq x_1 \leq x_2 \leq \cdots \leq x_m \leq
	n \right\},
	\]
	with an edge labelled by $i$ from any vertex whose $h$-th coordinate
	is $i$ to the vertex obtained by replacing that coordinate by $i+1$,
	whenever the resulting point is in $\mathcal{P}_{n,m}$. More
	generally, the connected component $\Gamma(\hypo_n,\sigma)$, for any
	composition $\sigma$ of $m$, is isomorphic to the graph obtained by
	translating this one in $\mathbb{Z}^m$ and relabelling so that any
	edge from a point in the hyperplane $x_h = i$ to one in $x_h = i+1$
	is labelled by $i$.
\end{proposition}

\begin{figure}[t]
  \begin{tikzpicture}

    \begin{scope}[
      bigcrystal,
      labelledcolouredcrystaledges,
      x={(16mm,20mm)},
      y={(18mm,0)},
      z={(0mm,-14mm)},
      ]
      \begin{scope}

        \node (111) at (1,1,1) {\smalltableau{1 \& 1 \\ \& 2 \\}};

        \node (112) at (1,1,2) {\smalltableau{1 \& 1 \\ \& 3 \\}};
        \node (122) at (1,2,2) {\smalltableau{1 \& 2 \\ \& 3 \\}};
        \node (222) at (2,2,2) {\smalltableau{2 \& 2 \\ \& 3 \\}};

        \node (113) at (1,1,3) {\smalltableau{1 \& 1 \\ \& 4 \\}};
        \node (123) at (1,2,3) {\smalltableau{1 \& 2 \\ \& 4 \\}};
        \node (223) at (2,2,3) {\smalltableau{2 \& 2 \\ \& 4 \\}};
        \node (133) at (1,3,3) {\smalltableau{1 \& 3 \\ \& 4 \\}};
        \node (233) at (2,3,3) {\smalltableau{2 \& 3 \\ \& 4 \\}};
        \node (333) at (3,3,3) {\smalltableau{3 \& 3 \\ \& 4 \\}};

        \node (114) at (1,1,4) {\smalltableau{1 \& 1 \\ \& 5 \\}};
        \node (124) at (1,2,4) {\smalltableau{1 \& 2 \\ \& 5 \\}};
        \node (224) at (2,2,4) {\smalltableau{2 \& 2 \\ \& 5 \\}};
        \node (134) at (1,3,4) {\smalltableau{1 \& 3 \\ \& 5 \\}};
        \node (234) at (2,3,4) {\smalltableau{2 \& 3 \\ \& 5 \\}};
        \node (334) at (3,3,4) {\smalltableau{3 \& 3 \\ \& 5 \\}};
        \node (144) at (1,4,4) {\smalltableau{1 \& 4 \\ \& 5 \\}};
        \node (244) at (2,4,4) {\smalltableau{2 \& 4 \\ \& 5 \\}};
        \node (344) at (3,4,4) {\smalltableau{3 \& 4 \\ \& 5 \\}};
        \node (444) at (4,4,4) {\smalltableau{4 \& 4 \\ \& 5 \\}};

        \draw[f2] (111) to (112);

        \draw[f3] (112) to (113);
        \draw[f3] (122) to (123);
        \draw[f3] (222) to (223);

        \draw[f4] (113) to (114);
        \draw[f4] (123) to (124);
        \draw[f4] (223) to (224);
        \draw[f4] (133) to (134);
        \draw[f4] (233) to (234);
        \draw[f4] (333) to (334);

        \draw[f1] (112) to (122);
        \draw[f1] (113) to (123);
        \draw[f1] (114) to (124);

        \draw[f2] (123) to (133);
        \draw[f2] (223) to (233);
        \draw[f2] (124) to (134);
        \draw[f2] (224) to (234);

        \draw[f3] (134) to (144);
        \draw[f3] (234) to (244);
        \draw[f3] (334) to (344);

        \draw[f1] (122) to (222);
        \draw[f1] (123) to (223);
        \draw[f1] (133) to (233);
        \draw[f1] (124) to (224);
        \draw[f1] (134) to (234);
        \draw[f1] (144) to (244);

        \draw[f2] (233) to (333);
        \draw[f2] (234) to (334);
        \draw[f2] (244) to (344);

        \draw[f3] (344) to (444);




      \end{scope}
    \end{scope}
  \end{tikzpicture}
  \caption{The connected component $\Gamma(\hypo_4,(2,1))$, drawn as a three-dimensional lattice, placing each entry of
    the quasi-ribbon tableau $\smalltableau{x_1 \& x_2 \\ \& x_3 \\}$ at the point $(x_1,x_2,x_3)$. As a lattice, this
    component is obtained from $\Gamma(\hypo_4,(3))$ by applying a shift of $+(0,0,1)$ and relabelling with $i$ any edge
    from a vertex in the plane with $x_3=i$ to one in the plane $x_3 = i+1$. (The relabelling applies to the edges drawn
    vertically downwards in this diagram.)}
  \label{fig:gamma-hypo-4-21}
\end{figure}

\section{Schur functions and fundamental quasi-symmetric function}\label{sec:schur_fund}

Maas-Gariépy conjectured \cite[Conjecture 2.6]{maasgariepy_quasicrystal} the following result:

\begin{theorem}\label{thm:schur-reorder}
  Let $\alpha$ be a composition and let $\lambda$ be the partition obtained by sorting the parts of $\alpha$ into weakly
  decreasing order. Then the fundamental quasi-symmetric function $F_\alpha$ occurs in the Schur function $s_\lambda$.
\end{theorem}

 Given the relationship between crystals and quasi-crystals, and the fact that their characters are, respectively,
Schur functions and fundamental quasi-symmetric functions, Theorem \ref{thm:schur-reorder} follows from the next statement about
quasi-crystals.

\begin{theorem}\label{thm:schur-fund}
  Let $\alpha$ be a composition and let $\lambda$ be the partition obtained by sorting the parts of $\alpha$ into weakly
  decreasing order. Any component of $\Gamma(\plac)$ comprising words whose tableaux have shape $\lambda$ contains a
  component of $\Gamma(\hypo)$ comprising words whose quasi-ribbon tableaux have shape $\alpha$.
\end{theorem}

\begin{proof}
  Let $\alpha$ and $\lambda$ be as in the statement. Let $Q$ be a highest weight quasi-ribbon tableau of shape
  $\alpha = (\alpha_1,\ldots,\alpha_{\clen{\alpha}})$. Then the $i$-th row of $Q$ comprises entries $i$, for
  $i \in \set{1,\ldots,\clen{\alpha}}$.

  Let $T$ be the array obtained from $Q$ by the following process: Slide all the rows leftwards until the leftmost entry
  of each is in column $1$. Now slide all the cells upwards along their columns until the topmost entry in each column
  is in the first row and there are no gaps in each column. The following is an example:
  \begin{align*}
    \begin{tikzpicture}[
      x=5mm,
      y=5mm,
      baseline=(firstmatrix-1-1.base),
      tableaumatrix/.append style={outer sep=0mm},
      ]
      \matrix[tableaumatrix,name=firstmatrix,anchor=north west]{
        1 \& 1                     \\
        \& 2 \& 2 \& 2             \\
        \&   \&   \& 3             \\
        \&   \&   \& 4 \& 4 \& 4 \& 4 \\
        \&   \&   \&   \&   \&   \& 5 \\
      };
    \end{tikzpicture}
    \rightsquigarrow{}&
    \begin{tikzpicture}[
      x=5mm,
      y=5mm,
      baseline=(firstmatrix-1-1.base),
      tableaumatrix/.append style={outer sep=0mm},
      ]
      \draw[gray,thick,->] (.6,-1.5) -- (0,-1.5);
      \draw[gray,thick,->] (1.8,-2.5) -- (0,-2.5);
      \draw[gray,thick,->] (1.8,-3.5) -- (0,-3.5);
      \draw[gray,thick,->] (2.4,-4.5) -- (0,-4.5);
      \matrix[tableaumatrix,name=firstmatrix,anchor=north west] at (0,0) {
        1 \& 1 \\
      };
      \matrix[tableaumatrix,anchor=north west] at (.6,-1) {
        2 \& 2 \& 2               \\
      };
      \matrix[tableaumatrix,anchor=north west] at (1.8,-2) {
        3       \\
      };
      \matrix[tableaumatrix,anchor=north west] at (1.8,-3) {
        4 \& 4 \& 4 \& 4 \\
      };
      \matrix[tableaumatrix,anchor=north west] at (2.4,-4) {
        5\\
      };
    \end{tikzpicture} \displaybreak[0]\\
    \rightsquigarrow{}&
    \begin{tikzpicture}[
      x=5mm,
      y=5mm,
      baseline=(firstmatrix-1-1.base),
      tableaumatrix/.append style={outer sep=0mm},
      ]
      \matrix[tableaumatrix,name=firstmatrix,anchor=north west] at (0,0) {
        1 \& 1 \\
        2 \& 2 \& 2 \\
        3 \\
        4 \& 4 \& 4 \& 4 \\
        5 \\
      };
    \end{tikzpicture} \displaybreak[0]\\
    \rightsquigarrow{}&
    \begin{tikzpicture}[
      x=5mm,
      y=5mm,
      baseline=(firstmatrix-1-1.base),
      tableaumatrix/.append style={outer sep=0mm},
      ]
      \draw[gray,thick,->] (1.5,-2.6) -- (1.5,-2);
      \draw[gray,thick,->] (2.5,-.6) -- (2.5,0);
      \draw[gray,thick,->] (2.5,-2.2) -- (2.5,-1.6);
      \draw[gray,thick,->] (3.5,-1.2) -- (3.5,0);
      \matrix[tableaumatrix,name=firstmatrix,anchor=north west] at (0,0) {
        1 \& 1 \\ 2 \& 2 \\ 3 \\ 4 \\ 5 \\
      };
      \matrix[tableaumatrix,anchor=north west] at (2,-.6) {
        2 \\
      };
      \matrix[tableaumatrix,anchor=north west] at (1,-2.6) {
        4 \\
      };
      \matrix[tableaumatrix,anchor=north west] at (2,-2.2) {
        4 \\
      };
      \matrix[tableaumatrix,anchor=north west] at (3,-1.2) {
        4 \\
      };
    \end{tikzpicture} \displaybreak[0]\\
    \rightsquigarrow{}&
    \begin{tikzpicture}[
      x=5mm,
      y=5mm,
      baseline=(firstmatrix-1-1.base),
      tableaumatrix/.append style={outer sep=0mm},
      ]
      \matrix[tableaumatrix,name=firstmatrix,anchor=north west] at (0,0) {
        1 \& 1 \& 2 \& 4 \\
        2 \& 2 \& 4 \\
        3 \& 4 \\
        4 \\
        5 \\
      };
    \end{tikzpicture}
  \end{align*}

  The result of sliding each row leftwards is an array $S$ where the $i$-th row has length $\alpha_i$ and contains only
  symbols $i$. Thus, although the columns of $S$ contain gaps, the entries in each column are strictly increasing from
  top to bottom. Therefore after sliding cells upwards to obtain $T$, the entries in each column are still strictly
  increasing from top to bottom. Further, if a column of $S$ has a gap in a given row, then so do all columns further to
  the right. Hence if an entry in $T$ comes from a given row of $S$, then each entry to its right in $T$ comes from the
  same row or a lower row of $S$. Since the $i$-th row of $S$ comprises only symbols $i$, it follows that each row of
  $T$ is weakly decreasing from left to right. Hence $T$ is actually a Young tableau.

  For any composition $\beta = (\beta_1,\ldots,\beta_{\clen\beta})$, define
  $\beta' = (\beta'_1,\ldots,\beta'_{\max(\beta)})$ by letting each $\beta'_j$ be the number of parts of $\beta$ that
  are greater than or equal to $j$. In the special case where $\beta$ is a partition, $\beta'$ is the conjugate
  partition and describes the number of cells in the columns of a Young tableau with shape $\beta$
  \cite[\S~1.3]{andrews_partitions}. It is immediate from the definition that $\beta'$ is invariant under the reordering
  of the parts of $\beta$. So, since the partition $\lambda$ is obtained by reordering the parts of $\alpha$ into weakly
  decreasing order, $\alpha' = \lambda'$. Further, each $\alpha_j$ is also the number of entries in the column $j$ of
  the array $S$, and the number of cells in each column does not change when cells are slid upwards. So $\alpha'$ is the
  conjugate of the shape of $T$. Hence, since partition conjugation is a bijection, $T$ has shape $\lambda$.

  Let $u$ be the column reading of $T$. Notice that, for $i \in \set{1,\ldots,\plen{\alpha}}$, the $i$-th row of $Q$
  comprises symbols $i$, and the $(i,1)$-th entry of $T$ is a symbol $i$. Thus $u$ has as a prefix
  $\plen{\alpha}\cdots 21$ and so contains an $(i+1)i$ inversion for all $i \in \set{1,\ldots,\plen{\alpha}-1}$.
  Furthermore, $Q$ and $T$, and so $u$, have the same evaluation. Hence $\phypo{u} = Q$.

  Thus the connected component $\Gamma(\hypo,u)$, which comprises words whose quasi-ribbon tableaux have shape $\alpha$,
  is contained in the connected component $\Gamma(\plac,u)$, which comprises words whose Young tableau have shape
  $\lambda$.
\end{proof}

{\color{black}
\begin{corollary}
Let $Q$ be a quasi-ribbon tableau and let $\xi (Q)$ denote the Young tableau obtained from $Q$ by applying the previous algorithm. Then, if $\f_i (Q)$ is defined, for some $i \in I$,
$$\xi (\f_i (Q)) = \kf_i (\xi (Q)).$$
\end{corollary}

\begin{proof}
Suppose that $\f_i (Q)$ is defined. Thus, $Q$ has a symbol $i$ and its column reading word has no $(i+1)i$ inversions. Equivalently, $Q$ has a symbol $i$ and the symbols $i+1$ that might occur in $Q$ must appear on the same row as $i$.

Therefore, the row reading word of $Q$ comprising the letters $i$ and $i+1$ is of the form $i^a (i+1)^b$, for $a > 0, b \geq 0$. Then, $\f_i$ changes the rightmost $i$ to $i+1$ and thus, the row reading word consisting of the letters $i$ and $i+1$ becomes $i^{a-1} (i+1)^{b+1}$. By construction, $\xi$ does not alter this word.

Now suppose that $\xi$ is first applied to $Q$. Thus, as before, the row reading word consisting of the letters $i$ and $i+1$ is not altered. Then, applying $\kf_i$ results in changing the rightmost $i$ in the row reading word of $\xi (Q)$ into $i+1$, becoming $i^{a-1} (i+1)^{b+1}$.

Therefore, $\xi (\f_i (Q))$ and $\kf_i (\xi (Q))$ have the same shape and the same row reading word, consisting of letters $i$ and $i+1$. Since the remaining letters were not changed during this process, we have $\xi (\f_i (Q)) = \kf_i (\xi (Q))$.
\end{proof}
}

\section{Skeletal structure}\label{sec:skeleton}

In \cite{maasgariepy_quasicrystal}, Maas-Gariépy introduced the notion of skeleton of a connected crystal, that we recall here.
Let $\lambda$ be a partition and let $n \in \nset$ be at least equal to the maximum length of a descent composition of a
standard Young tableau of shape $\lambda$. Take a connected component of $\Gamma(\plac_n)$ whose associated Young
tableaux have shape $\lambda$. Replace each quasi-crystal component it contains with the associated standard Young
tableau, keeping one oriented edge with minimal label of all those that joined each pair of adjacent quasi-crystal
components. The result is the \defterm{skeleton} of the original connected component of $\Gamma(\plac_n)$. It is
independent of the choice of $n$ \cite[Theorem~5]{maasgariepy_quasicrystal}. Since all connected components of
$\Gamma(\plac_n)$ whose associated Young tableaux have shape $\lambda$ are isomorphic, the skeleton is only dependent on
$\lambda$ and so is denoted $\Skel(\lambda)$.

Maas-Gariépy presented several conjectures about the skeleton of a connected crystal. In \cite[Conjecture 5.3]{maasgariepy_quasicrystal}, it is conjectured that the dual equivalence graph of a partition $\lambda$ is a subgraph of $\Skel(\lambda)$. This was recently proved by Brauner, Corteel, Daugherty and Schilling \cite[Theorem 4.1]{BCDS25}, although the definition of skeleton graph considered here differs slightly on the labelling of the edges.

Another conjecture concerns the structure of the induced subgraph $H_s$ of $\Skel(\lambda)$, whose
vertices are the standard Young tableaux with descent compositions having exactly $s$ parts. The author gave some conjectures based
on computations when $\lambda$ is a partition of a number $n \leq 6$: that each $H_s$ is a disjoint union of
singletons; a disjoint union of chains; or a union of cycles possibly with an extra source and sink vertex
attached \cite[Conjecture~4.10]{maasgariepy_quasicrystal}. However, these conjectures largely fail for the following example, where $\lambda$ is a partition of $7$.

\begin{example}\label{ex:counter-example}
  Let $\lambda = (3,2,2)$. There are $21$ standard Young tableaux of this shape, and the maximum length of any of their
  descent compositions is $5$. The skeleton $\Skel(\lambda)$ is as shown in \fullref{Figure}{fig:skel-322}.
  \begin{itemize}
    \item The $3$ tableaux at the top have descent compositions with $3$ parts, and the subgraph $H_3$ they induce is a
          chain.
    \item The $6$ tableaux at the bottom have descent compositions with $5$ parts, and the subgraph $H_5$ they induce is
          a union of a chain and a singleton.
    \item The remaining $12$ tableaux have descent compositions with $4$ parts, and the subgraph $H_4$ they induce
          contains two cycles of even length and four additional vertices.
  \end{itemize}
  The subgraphs $H_4$ and $H_5$ suffice to show that the conjecture does not hold.

  \begin{figure}[t]
    \begin{tikzpicture}[
          bigcrystal,
          labelledcolouredcrystaledges
        ]

      \node (6417523) at (-1,4) {\smalltableau{1 \& 2 \& 3 \\ 4 \& 5 \\ 6 \& 7 \\}};
      \node (6317425) at (0,4) {\smalltableau{1 \& 2 \& 5 \\ 3 \& 4 \\ 6 \& 7 \\}};
      \node (5316427) at (1,4) {\smalltableau{1 \& 2 \& 7 \\ 3 \& 4 \\ 5 \& 6 \\}};

      \node (5317426) at (2,2) {\smalltableau{1 \& 2 \& 6 \\ 3 \& 4 \\ 5 \& 7 \\}};
      \node (4317526) at (2,1) {\smalltableau{1 \& 2 \& 6 \\ 3 \& 5 \\ 4 \& 7 \\}};
      \node (4316527) at (3,0) {\smalltableau{1 \& 2 \& 7 \\ 3 \& 5 \\ 4 \& 6 \\}};
      \node (4217536) at (1,0) {\smalltableau{1 \& 3 \& 6 \\ 2 \& 5 \\ 4 \& 7 \\}};
      \node (4216537) at (2,-1) {\smalltableau{1 \& 3 \& 7 \\ 2 \& 5 \\ 4 \& 6 \\}};
      \node (5216437) at (2,-2) {\smalltableau{1 \& 3 \& 7 \\ 2 \& 4 \\ 5 \& 6 \\}};

      \node (5417623) at (-2,2) {\smalltableau{1 \& 2 \& 3 \\ 4 \& 6 \\ 5 \& 7 \\}};
      \node (5317624) at (-2,1) {\smalltableau{1 \& 2 \& 4 \\ 3 \& 6 \\ 5 \& 7 \\}};
      \node (5217634) at (-1,0) {\smalltableau{1 \& 3 \& 4 \\ 2 \& 6 \\ 5 \& 7 \\}};
      \node (6317524) at (-3,0) {\smalltableau{1 \& 2 \& 4 \\ 3 \& 5 \\ 6 \& 7 \\}};
      \node (6217534) at (-2,-1) {\smalltableau{1 \& 3 \& 4 \\ 2 \& 5 \\ 6 \& 7 \\}};
      \node (6217435) at (-2,-2) {\smalltableau{1 \& 3 \& 5 \\ 2 \& 4 \\ 6 \& 7 \\}};

      \node (3216547) at (2,-4) {\smalltableau{1 \& 4 \& 7 \\ 2 \& 5 \\ 3 \& 6 \\}};
      \node (3217546) at (1,-4) {\smalltableau{1 \& 4 \& 6 \\ 2 \& 5 \\ 3 \& 7 \\}};
      \node (3217645) at (0,-4) {\smalltableau{1 \& 4 \& 5 \\ 2 \& 6 \\ 3 \& 7 \\}};
      \node (4217635) at (-1,-4) {\smalltableau{1 \& 3 \& 5 \\ 2 \& 6 \\ 4 \& 7 \\}};
      \node (4317625) at (-2,-4) {\smalltableau{1 \& 2 \& 5 \\ 3 \& 6 \\ 4 \& 7 \\}};

      \node (5217436) at (0,-3) {\smalltableau{1 \& 3 \& 6 \\ 2 \& 4 \\ 5 \& 7 \\}};

      \draw[f1] (6417523) to (6317425);
      \draw[f2] (6317425) to (5316427);

      \draw[f2] (5317426) to (4317526);
      \draw[f3] (4317526) to (4316527);
      \draw[f1] (4317526) to (4217536);
      \draw[f1] (4316527) to (4216537);
      \draw[f3] (4217536) to (4216537);
      \draw[f3] (4216537) to (5216437);
      \draw[f1] (5417623) to (5317624);
      \draw[f1] (5317624) to (5217634);
      \draw[f3] (5317624) to (6317524);
      \draw[f3] (5217634) to (6217534);
      \draw[f1] (6317524) to (6217534);
      \draw[f2] (6217534) to (6217435);
      \draw[f2] (5217634) to (4217536);

      \draw[f1] (4317625) to (4217635);
      \draw[f2] (4217635) to (3217645);
      \draw[f3] (3217645) to (3217546);
      \draw[f4] (3217546) to (3216547);

      \draw[f2,bend right=10,swap] (6417523) to (5417623);
      \draw[f3,bend right=10,swap] (5417623) to (6417523);
      \draw[f1,swap] (6417523) to[out=170,in=100] (6317524);
      \draw[f3,swap] (6317524) to[out=90,in=180] (6417523);

      \draw[f1] (6317425) to[out=260,in=0,pos=.25] (6217435);
      \draw[f3] (6217435) to[out=10,in=250,pos=.75] (6317425);
      \draw[f2] (6317425) to[out=-80,in=170,pos=.8] (5317426);
      \draw[f4] (5317426) to[out=180,in=-90,pos=.2] (6317425);

      \draw[f2] (5316427) to[out=10,in=80] (4316527);
      \draw[f4] (4316527) to[out=90,in=0] (5316427);
      \draw[f1] (5316427) to[out=260,in=170,pos=.25] (5216437);
      \draw[f3] (5216437) to[out=180,in=250,pos=.75] (5316427);

      \draw[f1] (5217436) to[out=100,in=220,pos=.75] (5317426);
      \draw[f2] (5317426) to[out=230,in=90,pos=.25] (5217436);
      \draw[f3,bend right=10,swap,pos=.25] (5217436) to (6217435);
      \draw[f4,bend right=10,swap,pos=.75] (6217435) to (5217436);

      \draw[f2,swap] (4317625) to[out=150,in=190] (5317624);
      \draw[f3,swap] (5317624) to[out=180,in=160] (4317625);

      \draw[f2,bend right=5,swap,pos=.15] (4217635) to (5217634);
      \draw[f3,bend right=5,swap,pos=.85] (5217634) to (4217635);

      \draw[f2,bend right=5,swap,pos=.15] (3217546) to (4217536);
      \draw[f3,bend right=5,swap,pos=.85] (4217536) to (3217546);

      \draw[f2] (3216547) to[out=30,in=-10] (4216537);
      \draw[f3] (4216537) to[out=0,in=20] (3216547);

    \end{tikzpicture}
    \caption{$\Skel(\lambda)$, where $\lambda = (3,2,2)$.}
    \label{fig:skel-322}
  \end{figure}
\end{example}

One part of the conjecture is true:

\begin{proposition}\label{prop:even-len}
  Let $\lambda$ be a partition and let $H_s$ be the subgraph of $\Skel(\lambda)$ induced by the standard Young tableaux
  whose descent compositions have exactly $s$ parts. Then any cycle in $H_s$ has even length.
\end{proposition}

\begin{proof}
  By the proof of \cite[Theorem~5]{maasgariepy_quasicrystal}, if $T$ and $T'$ are Young tableaux with descent
  compositions $\alpha$ and $\alpha'$ with the same number of parts, and $\f_i(T) = T'$, then $\alpha'$ can be obtained
  from $\alpha$ by decrementing one part by $1$ and incrementing an adjacent part by $1$. In terms of quasi-ribbon
  diagrams, this corresponds to moving a cell from the $i$-th row to the $(i+1)$-th or vice versa.

  Define the \defterm{parity} of a Young tableaux with descent composition
  $\alpha = \parens{\alpha_1,\ldots,\alpha_{\plen{\alpha}}}$ to be the partity of
  $\sum_{j=1}^{\floor{\plen{\alpha}/2}} \alpha_{2j}$; that is, the parity of the sum of the even-indexed parts of
  $\alpha$. By the previous paragraph, any two Young tableaux that are adjacent in $H_s$ will have opposite parities.
  Hence any path from a vertex of $H_s$ back to that vertex must have even length.
\end{proof}

\section{Acknowledgments}
The authors are thankful to the anonymous referees for their careful reading of the paper and several helpful comments.

\bibliographystyle{alpha}

\bibliography{cmrr_quasicrystals_2025-05-30.bib}

\end{document}